\documentclass[a4paper]{amsart}
\usepackage{a4wide}
\usepackage[english]{babel}
\usepackage[utf8x]{inputenc}
\usepackage{amsmath,amssymb,amsthm}
\usepackage[table]{xcolor}
\usepackage{graphicx,enumerate,url}
\usepackage[colorinlistoftodos]{todonotes}

\usepackage{tikz}
\usetikzlibrary{calc,through,backgrounds,shapes,matrix}

\usepackage{array}
\newcolumntype{C}[1]{>{\centering\arraybackslash$}p{#1}<{$}}
\usepackage{multirow}
\usepackage{hhline}
\newcommand{\onevar}{\mathbf{1}}
\newcommand{\zerovar}{\mathbf{0}}

\usepackage{hyperref}

\usepackage{url, hypcap}
\hypersetup{colorlinks=true, citecolor=darkblue, linkcolor=darkblue}

\numberwithin{equation}{section}

\usepackage{cleveref}
\crefname{conjecture}{conjecture}{conjectures}
\Crefname{conjecture}{Conjecture}{Conjectures}
\crefname{observation}{observation}{observations}
\Crefname{observation}{Observation}{Observations}
\crefname{hope}{hope}{hopes}
\Crefname{hope}{Hope}{Hopes}

\usepackage[colorinlistoftodos]{todonotes}
\newtheorem{theorem}{Theorem}[section]
\newtheorem{proposition}[theorem]{Proposition}
\newtheorem{lemma}[theorem]{Lemma}
\newtheorem{corollary}[theorem]{Corollary}

\newtheorem{conjecture}[theorem]{Conjecture}

\theoremstyle{definition}
\newtheorem{example}[theorem]{Example}
\newtheorem{remark}[theorem]{Remark}

\newcommand{\ssm}{\smallsetminus} %

\DeclareMathOperator{\conv}{conv} %

\DeclareMathOperator{\Trop}{Trop}
\DeclareMathOperator{\Tropplus}{Trop^+}
\DeclareMathOperator{\spec}{Spec}

\newcommand{\ZZ}{\mathbb{Z}}
\newcommand{\NN}{\mathbb{N}}
\newcommand{\RR}{\mathbb{R}}
\newcommand{\QQ}{\mathbb{Q}}

\newcommand{\Sym}{{\mathfrak{S}}}

\newcommand{\Phiplus}{\Phi^+}
\newcommand{\Phipm}{\Phi_{\geq -1}}

\definecolor{darkblue}{rgb}{0,0,0.7} %

\definecolor{lightgrey}{rgb}{0.9,0.9,0.9} %

\definecolor{grey}{rgb}{0.5,0.5,0.5} %

\newcommand{\Dfn}[1]{\emph{\bfseries #1}} %

\newcommand{\x}{\mathbf{x}}
\newcommand{\y}{\mathbf{y}}
\renewcommand{\u}{\mathbf{u}}

\newcommand{\Alg}{\mathcal{A}}

\newcommand{\subwordComplex}{\mathcal{SC}} %
\newcommand{\clusterComplex}{\subwordComplex\big(\cw{c}\big)}
\newcommand{\greedy}{I_{\operatorname{g}}} %
\newcommand{\antigreedy}{I_{\operatorname{ag}}} %

\newcommand{\sq}[1]{{\rm #1}} %
\newcommand{\Q}{\sq{Q}} %
\newcommand{\q}{\sq{q}} %
\newcommand{\s}{\sq{s}} %

\newcommand{\sref}{\mathcal{S}}

\newcommand{\sfr}{{\sf r}} %
\newcommand{\sfR}{{\sf R}} %
\newcommand{\sfw}{{\sf w}} %

\newcommand{\wo}{{w_\circ}}

\renewcommand{\c}{\sq{c}}
\newcommand{\sw}[2]{\sq{#1}(\sq{#2})} %
\newcommand{\cwo}[1]{\sw{\wo}{#1}} %
\newcommand{\cw}[1]{\sq{#1}\cwo{#1}} %

\newcommand{\brickVector}{{\sf b}} %
\newcommand{\Root}[2]{{\sfr}(#1,#2)} %
\newcommand{\Roots}[1]{{\sfR}(#1)} %
\newcommand{\Weight}[2]{{\sfw}(#1,#2)} %

\newcommand{\Newton}[1]{{\operatorname{Newton}\left(#1\right)}}

\newcommand{\set}[2]{\left\{ #1 \;|\; #2 \right\}} %
\newcommand{\bigset}[2]{\left\{ #1 \;\big|\; #2 \right\}} %
\newcommand{\bigmultiset}[2]{\big\{\!\!\big\{ #1 \;\big|\; #2 \big\}\!\!\big\}} %
\newcommand{\bigmultibrackets}[1]{\big\{\!\!\big\{ #1 \big\}\!\!\big\}}

\newcommand{\wordprod}[2]{\Pi{#1}_{#2}} %

\newcommand{\sfM}{{\sf M}} %
\newcommand{\mutmatrix}{\sfM} %
\newcommand{\Asso}{{\sf Asso}}
\DeclareMathOperator{\Assoc}{{{\Asso\left(\mutmatrix_c\right)}}}
\DeclareMathOperator{\Assocbeta}{{{\Asso_{\beta}\left(\mutmatrix_c\right)}}}

\newcommand{\gFan}{{\mathcal F_g}}

\newcommand{\mii}[1]{\overline{#1}}
\newcommand{\mi}[1]{$\overline{#1}$}
\newcommand{\D}[1]{$#1_\Delta$}

\def\ClustComp{S}
\def\fan{\mathcal F}
\newcommand{\TC}[1]{\mathbb{TC}(#1)}
\newcommand{\bigTC}[1]{\mathbb{TC}\big(#1\big)}

\newcommand{\TCclosed}[1]{\overline{\mathbb{TC}}(#1)}

\DeclareMathOperator{\cone}{cone}

\def\RR{\mathbb R}
\def\ZZ{\mathbb Z}
\def\field{\QQ}

\def\ideal{\mathcal I}
\def\variety{V}
\def\weight{w}

\title[Minkowski decompositions for generalized associahedra]{Minkowski decompositions for\\ generalized associahedra of acyclic type}

\author[D.~Jahn]{Dennis Jahn}
\address[D.~Jahn]{Fakultät für Mathematik, Ruhr-Universität Bochum, Germany}
\email{dennis.jahn@rub.de}

\author[R.~L{\"o}we]{Robert L{\"o}we}
\address[R.~L{\"o}we]{Institut f{\"u}r Mathematik, TU Berlin, Germany}
\email{loewe@math.tu-berlin.de}

\author[C.~Stump]{Christian Stump}
\address[C.~Stump]{Fakultät für Mathematik, Ruhr-Universität Bochum, Germany}
\email{christian.stump@rub.de}

\thanks{
  R.L.~is supported by the DFG grant SFB-TRR 109 ``Discretization in Geometry and Dynamics''.
  D.J.~and C.S.~are supported by the DFG Heisenberg grant STU 563/4-1 ``Noncrossing phenomena in Algebra and Geometry''.
}

\begin{document}

\begin{abstract}
  We give an explicit subword complex description of the generators of the type cone of the $g$-vector fan of a finite type cluster algebra with acyclic initial seed.
  This yields in particular a description of the Newton polytopes of the $F$-polynomials in terms of subword complexes as conjectured by S.~Brodsky and the third author.
  We then show that the cluster complex is combinatorially isomorphic to the totally positive part of the tropicalization of the cluster variety as conjectured by D.~Speyer and L.~Williams.
\end{abstract}

\maketitle

\section{Introduction and main results}
\label{sec:intro}

A generalized associahedron for a cluster algebra of finite type is a simple polytope whose face lattice is dual to the cluster complex.
Constructing such generalized associahedra has been a fruitful area of mathematical research  since the introduction of cluster algebras by S.~Fomin and A.~Zelevinsky in the early 2000s.
We refer to \cite{FZ2002, CFZ2002,HLT2008,Pilaud-Stump-2015,HPS2018} in this chronological order for some of the milestones and history.
This paper is a continuation of~\cite{BS-2018} and builds on recent results from~\cite{ABHY-2018, AHHL2020} and from~\cite{Padrol-Pilaud-2019}.

\medskip

The paper has three major results, two of which resolve conjectures by S.~Brodsky and the third author and, respectively, by D.~Speyer and L.~Williams.
\Cref{thm:indecomp_columns} gives a self-contained combinatorial construction of the rays of the type cone of the $g$-vector fan of a finite type cluster algebra with acyclic initial seed via subword complexes and brick polytopes.
Using this construction together with recent results from~\cite{ABHY-2018, AHHL2020} and~\cite{Padrol-Pilaud-2019}, \Cref{cor:NewtonFpoly} yields that this construction also describes the Newton polytopes of the $F$-polynomials of the cluster algebra.
This description was conjectured in~\cite[Conjecture~2.12]{BS-2018}.
The appearance of the $F$-polynomials is then as well used to derive \Cref{thm:tropplusvariety} showing that the totally positive part of the tropical cluster variety is, modulo its lineality space, linearly isomorphic to the $g$-vector fan.
As the $g$-vector fan is combinatorially isomorphic to the cluster complex, this affirmatively answers~\cite[Conjecture~8.1]{Speyer-Williams-2005} for finite type cluster algebras with principal coefficients and acyclic initial seed.

\medskip

In order to precisely state the results, let~$\Delta \subseteq \Phiplus \subseteq \Phipm \subseteq \Phi$ denote a finite crystallographic root system with fundamental weights~$\nabla$ and let $\mutmatrix$ denote an initial mutation matrix with principal coefficients for a cluster algebra $\Alg(\mutmatrix)$ of type~$\Phi$ with cluster variables $\set{u_\beta(\x,\y)}{\beta \in \Phipm}$ and cluster complex $\ClustComp(\mutmatrix)$ given by the set of compatible cluster variables.
The cluster variables have the form $u_\beta(\x,\y) = p(\x,\y) / \x^\beta$ with $\x = (x_1,\dots,x_n)$ and $\y = (y_1,\dots,y_n)$ for $p(\x,\y) \in \NN[\x,\y]$ and $\x^\beta = x_1^{\beta_1} \cdots x_n^{\beta_n}$ with $\beta = \beta_1\alpha_1+\dots+\beta_n\alpha_n$ expanded in the root basis $\Delta = \{\alpha_1,\dots,\alpha_n\}$.
Its $F$-polynomials are denoted by $\bigset{F_\beta = u_\beta(\onevar,\y)}{\beta \in \Phiplus}$ and its $g$-vector fan $\gFan(\mutmatrix)$ is given by the cones over compatible sets of $g$-vectors $g_\beta = g_1\omega_1+\dots+g_n\omega_n$ such that $u_\beta(\x,\zerovar) = x_1^{g_1}\cdots x_n^{g_n}$ expanded in the weight basis~$\nabla = \{\omega_1,\dots,\omega_n\}$.
It is well-known that the $g$-vector fan is combinatorially isomorphic to the cluster complex $S(\mutmatrix)$.
Let~$(W,\sref)$ denote the Coxeter system generated by $\sref=\set{s_\alpha}{\alpha \in \Delta}$ and let $c \in W$ be a standard Coxeter element given by the product of the reflections in $\sref$ in some order.
One may associate to this data an acyclic initial mutation matrix $\mutmatrix_\c$ with principal coefficients, and as well a brick polytope $\Assoc$ with normal fan given by the $g$-vector fan $\gFan(\mutmatrix_\c)$ of $\Alg(\mutmatrix_\c)$.
In particular, $\Assoc$ is a generalized associahedron for $\Alg(\mutmatrix_\c)$.
Brick polytopes for subword complexes come with natural Minkowski decompositions which in the present context may be written in the form
\begin{equation}
\label{eq:natural_mink_dec}
\Assoc = \sum_{\beta \in \Phiplus} \Assocbeta.
\end{equation}
The type cone $\TC{\gFan(\mutmatrix_\c)}$ of the $g$-vector fan is the space of all its polytopal realizations.
We thus have
\begin{equation}
\Assoc \in \bigTC{\gFan(\mutmatrix_\c)}. \label{eq:Associntypecone}
\end{equation}
While motivated by beautiful constructions in~\cite{ABHY-2018} and~\cite{Padrol-Pilaud-2019}, the following result is entirely self-contained and only uses properties of brick polytopes developed in~\cite{Pilaud-Stump-2015} and~\cite{BS-2018}.

\begin{theorem}
  \label{thm:indecomp_columns}
  For an acyclic initial mutation matrix $\mutmatrix_\c$ with principal coefficients, the type cone of the $g$-vector fan $\gFan(\mutmatrix_\c)$ is the open simplicial cone
  generated by the natural Minkowski summands of the brick polytope $\Assoc$,
  \[
  \bigTC{\gFan(\mutmatrix_\c)} \ = \ \cone\bigset{\Assocbeta}{\beta \in \Phiplus}.
  \]
\end{theorem}

\begin{remark}
  This theorem and its proof are combinatorial and do not use any representation theory.
  The definition of a generalized associahedron $\Assoc$ in~\cite{HLT2008,Pilaud-Stump-2015} extends verbatim to the noncrystallographic finite types $I_2(m)$ for $m \notin \{3,4,6\}$ and $H_3,H_4$.
  The theorem also holds for noncrystallographic types when replacing the left-hand side by the type cone of weak Minkowski summands of $\Assoc$ even though mutations of cluster variables, $g$-vectors and $F$-polynomials in these types do not behave combinatorially nicely~\cite{Lam2018}.
\end{remark}

Combining~\cite[Theorem~3]{ABHY-2018} (simply-laced types) and~\cite[Theorem~6.1]{AHHL2020} (multiply-laced types) with~\cite[Theorem~2.26]{Padrol-Pilaud-2019}, one obtains that the rays of the type cone of the $g$-vector fan are also equal to the Newton polytopes of the $F$-polynomials,
\begin{equation}
  \bigTC{\gFan(\mutmatrix_\c)} \ = \ \cone\bigset{\Newton{F_\beta}}{\beta \in \Phiplus},\tag{$\star$} \label{eq:newtonF}
\end{equation}
where the exponent vectors are written in the root basis~$\Delta$, and in particular that
\begin{equation}
  \sum_{\beta \in \Phiplus} \Newton{F_\beta} \ \in \ \bigTC{\gFan(\mutmatrix_\c)} \tag{$\star\star$} \label{eq:newtonF2}
\end{equation}
is a generalized associahedron for $\Alg(\mutmatrix_\c)$.
According to~\cite{Padrol-Pilaud-2019}, H.~Thomas announced that a future version of~\cite{ABHY-2018} will generalize~\eqref{eq:newtonF} also to cyclic finite types.
In this case,~\eqref{eq:newtonF2} was conjectured by S.~Brodsky and the third author in~\cite[Conjecture~2.22]{BS-2018}.
Combining this with \Cref{thm:indecomp_columns} and known properties of $F$-polynomials, we obtain the second main result describing Newton polytopes of $F$-polynomials for acyclic initial seeds in terms of subword complexes.

\begin{theorem}[{\cite[Conjecture~2.12]{BS-2018}}]
  \label{cor:NewtonFpoly}
  Let $\mutmatrix_\c$ be an acyclic initial mutation matrix with principal coefficients.
  For any positive root $\beta \in \Phi^+$, we have 
  \[
  \Newton{F_\beta} \ = \ \Assocbeta.
  \]
\end{theorem}

In \cite{Speyer-Williams-2005} the authors associate to the cluster algebra $\Alg(\mutmatrix)$ a polyhedral fan $\Tropplus \spec \Alg(\mutmatrix)$ by tropicalizing the positive part of the affine variety $\spec \Alg(\mutmatrix)$.
Using~\eqref{eq:newtonF2}, we finally derive the following theorem\footnote{In response to a first preprint, Thomas Lam informed us that a more general version of this theorem also follows from~\cite[Theorems~4.1 \& 4.2]{AHHL2020} which implies the first part of \cite[Conjecture~8.1]{Speyer-Williams-2005}.}.

\begin{theorem}
  \label{thm:tropplusvariety}
  For acyclic initial mutation matrix $\mutmatrix_\c$ with principal coefficients, the totally positive part of the tropical variety associated to the cluster algebra $\Alg(\mutmatrix_\c)$ is, modulo its lineality space~$\mathcal L$, linearly isomorphic to the $g$-vector fan,
  \[
  \Tropplus\spec{\Alg(\mutmatrix_\c)}\big/ \mathcal L \ \cong \ \gFan(\mutmatrix_\c).
  \]
\end{theorem}

As the $g$-vector fan is combinatorially isomorphic to the cluster complex, this affirmatively answers a conjecture by D.~Speyer and L.~Williams in this situation.

\begin{corollary}[{\cite[Conjecture~8.1]{Speyer-Williams-2005}}]
  In the situation of \Cref{thm:tropplusvariety}, the cluster complex $\ClustComp(\mutmatrix)$ is combinatorially isomorphic to the polyhedral fan~$\Tropplus\spec{\Alg(\mutmatrix_\c)}$ .
\end{corollary}

\subsection{Acknowledgements}

The third author would like to thank Thomas Lam, Arnau Padrol, Markus Reineke, Raman Sanyal and Hugh Thomas for valuable discussions concerning various parts of this paper.

\section{A natural Minkowski decomposition of generalized associahedra}
\label{sec:defintionsresults}

We follow the notions from~\cite{BS-2018} and refer to Section~2 therein for details.

\subsection{Generalized associahedra for acyclic type}

Let $(W,\sref)$ be a finite type Coxeter system of rank~$n$ and let $\Delta \subseteq \Phiplus \subseteq \Phipm \subseteq \Phi \subseteq V$ be a finite root system for $(W,\sref)$ inside an Euclidean vector space~$V$, with simple roots $\Delta = \set{\alpha_s}{s \in \sref}$, positive roots~$\Phiplus$ and almost positive roots $\Phipm = \Phiplus \sqcup -\Delta$.
Denote by $N = |\Phiplus|$ the number of positive roots and $n+N = |\Phipm|$.
Let~$C = (a_{st})_{s,t \in \sref}$ denote the corresponding \Dfn{Cartan matrix} given by $s(\alpha_t) = \alpha_t - a_{st} \alpha_s$ and set $\nabla = \set{\omega_s}{s \in \sref} \subseteq V$ to be the fundamental weights given by
\begin{align*}
  \alpha_s = \sum_{t \in \sref}a_{ts}\omega_t.
\end{align*}
One then has $s(\omega_t) = \omega_t - \delta_{s=t}\alpha_s$ for $s,t \in \sref$.
In all below considerations we consider~$V \cong \RR^\Delta$ to have fixed basis~$\Delta$, though in the examples we simultaneously consider the vector space with standard basis and standard inner product.

We consider a fixed Coxeter element~$c \in W$ and a reduced word $\c = \s_1\cdots\s_n$ for~$c$.
To avoid double indices we write $\alpha_i$ for $\alpha_{s_i}$ and $\omega_i = \omega_{s_i}$.
The initial mutation matrix $\mutmatrix_\c = (m_{ij})$ is then obtained from the Cartan matrix by
\[
  m_{ij} = \begin{cases}
             0      &\text{ if } i=j \\
             -a_{st} &\text{ if } s=s_i \text{ appears before } t=s_j \text{ in the reduced word } \c \\
              a_{st} &\text{ if } s=s_i \text{ appears after  } t=s_j \text{ in the reduced word } \c
           \end{cases},
\]
for $1 \leq i,j \leq n$, together with an identity matrix below.

\begin{figure}[t]
  \setlength{\extrarowheight}{3pt}
  \setlength{\tabcolsep}{5pt}
  \begin{tabular}{ l | c | c | l }
    cluster variable & $d$-vector & $g$-vector & $F$-polynomial \\
    \hline
    \hline
    \multicolumn{4}{l}{Example $A_3$}\\
    $x_1$  & $\mii{1}00_\Delta$ & $100_\nabla = \tfrac{1}{4} 321_\Delta$ & \\
    \hline
    $x_2$  & $0\mii{1}0_\Delta$ & $010_\nabla = \tfrac{1}{4} 121_\Delta$ & \\
    \hline
    $x_3$  & $00\mii{1}_\Delta$ & $001_\nabla = \tfrac{1}{4} 123_\Delta$ & \\
    \hline
    $x_4 = \frac{x_2 + y_1}{x_1}$ & $100_\Delta$ & $\mii{1}10_\nabla = \tfrac{1}{4} \mii{2}00_\Delta$ & $F_{100_\Delta} = y_1 + 1$ \\
    \hline
    $x_5 = \frac{x_1y_1y_2 + x_3y_1 + x_2x_3}{x_1x_2}$ & $110_\Delta$ & $\mii{1}01_\nabla = \tfrac{1}{4} \mii{2}02_\Delta$ & $F_{110_\Delta} = y_1y_2 + y_1 + 1$ \\
    \hline
    $x_6 = \frac{x_1x_2y_1y_2y_3 + x_1y_1y_2 + x_3y_1 + x_2x_3}{x_1x_2x_3}$ & $111_\Delta$ & $\mii{1}00_\nabla = \tfrac{1}{4} \mii{321}_\Delta$ & $F_{111_\Delta} = y_1y_2y_3 + y_1y_2 + y_1 + 1$ \\
    \hline
    $x_7 = \frac{x_1y_2 + x_3}{x_2}$ & $010_\Delta$ & $0\mii{1}1_\nabla = \tfrac{1}{4} 002_\Delta$ & $F_{010_\Delta} = y_2 + 1$ \\
    \hline
    $x_8 = \frac{x_1x_2y_2y_3 + x_1y_2 + x_3}{x_2x_3}$ & $011_\Delta$ & $0\mii{1}0_\nabla = \tfrac{1}{4} \mii{121}_\Delta$ & $F_{011_\Delta} = y_2y_3 + y_2 + 1$ \\
    \hline
    $x_9 = \frac{x_2y_3 + 1}{x_3}$ & $001_\Delta$ & $00\mii{1}_\nabla = \tfrac{1}{4} \mii{123}_\Delta$ & $F_{001_\Delta} = y_3 + 1$ \\
    \hline
    \hline
    \multicolumn{4}{l}{Example $B_2$}\\
    $x_1$ & $\mii{1}0_\Delta$ & $10_\nabla = \hspace{6pt} 11_\Delta$ & \\
    \hline
    $x_2$ & $0\mii{1}_\Delta$ & $01_\nabla = \tfrac{1}{2} 12_\Delta$ & \\
    \hline
    $x_3 = \frac{x_2^2 + y_1}{x_1}$ & $10_\Delta$ & $\mii{1}2_\nabla = \hspace{6pt} 01_\Delta$ & $F_{10_\Delta} = y_1 + 1$ \\
    \hline
    $x_4 = \frac{x_1y_1y_2 + y_1 + x_2^2}{x_1x_2}$ & $11_\Delta$ & $\mii{1}1_\nabla = \tfrac{1}{2} \mii{1}0_\Delta$ & $F_{11_\Delta} = y_1y_2 + y_1 + 1$ \\
    \hline
    $x_5 = \frac{x_1^2y_1y_2^2 + 2x_1y_1y_2 + y_1 + x_2^2}{x_1x_2^2}$ & $12_\Delta$ & $\mii{1}0_\nabla = \hspace{6pt} \mii{11}_\Delta$ & $F_{12_\Delta} = y_1y_2^2 + 2y_1y_2 + y_1 + 1$ \\
    \hline
    $x_6 = \frac{x_1y_2 + 1}{x_2}$ & $01_\Delta$ & $0\mii{1}_\nabla = \tfrac{1}{2} \mii{12}_\Delta$ & $F_{01_\Delta} = y_2 + 1$
  \end{tabular}
  \caption{The cluster variables with its $d$-vectors, $g$-vectors and $F$-polynomials for the initial mutation matrices in \Cref{ex:A3mutmatrix_2,,ex:B2mutmatrix_2}.}
  \label{fig:examples}
\end{figure}

\begin{example}[$A_3$-example]
  \label{ex:A3mutmatrix_2}
  Take $W = \Sym_4$ the symmetric group with adjacent transpositions as simple generators
  \[
    \sref = \big\{s_1 = (1,2),\ s_2 = (2,3),\ s_3 = (3,4) \big\}
  \]
  acting on $V = \set{(\lambda_1,\lambda_2,\lambda_3,\lambda_4) \in \RR^4}{\lambda_1+\lambda_2+\lambda_3+\lambda_4 = 0}\cong \RR^4 / \RR (1111)$, equipped with the standard inner product, by permuting the standard basis.
  We choose
  \[
    \Delta = \big\{ \alpha_1 = 1\mii{1}00,~ \alpha_2 = 01\mii{1}0,~ \alpha_3 = 001\mii{1} \big\}
  \]
  as a basis of~$V$.
  Here and below we write shorthand $\mii \lambda := -\lambda$ for scalars~$\lambda$.
  We may express an element in~$V$ in round brackets such as $(1,0,\mii 1,0) = (1,1,0)_\Delta = \alpha_1+\alpha_2$ where the first expression is the genuine element in~$V \subset \RR^4$ and the second expression is in the chosen basis $\Delta$ in the given order.
  We obtain
  \begin{align*}
  \Phiplus &= \big\{  1\mii{1}00\hspace{2px},\ 
  10\mii{1}0\hspace{2px},\ 
  100\mii{1}\hspace{2px},\ 
  01\mii{1}0\hspace{2px},\ 
  010\mii{1}\hspace{2px},\ 
  001\mii{1}\hspace{2px}  \big\} \\
  &= \big\{  100_\Delta,\ 
  110_\Delta,\ 
  111_\Delta,\ 
  010_\Delta,\ 
  011_\Delta,\ 
  001_\Delta \big\},
  \end{align*}
  and finally $\Phipm = \Phiplus \sqcup -\Delta$ and $\Phi = \Phiplus \sqcup -\Phiplus$.
  In this case, $n = |\sref| = 3$ and $N=|\Phiplus| = 6$.
  The corresponding Cartan matrix is
  \[
  C =
  \begin{pmatrix}
  2       & \mii{1} & 0       \\
  \mii{1} & 2       & \mii{1} \\
  0       & \mii{1} & 2
  \end{pmatrix}
  \]
  and the fundamental weights are
  \[
  \nabla = \big\{ \omega_1 = 1000 = \tfrac{1}{4}(321_\Delta),\quad
  \omega_2 = 1100 = \tfrac{1}{4}(121_\Delta),\quad
  \omega_3 = 1110 = \tfrac{1}{4}(123_\Delta) \big\}.
  \]
  Fix the Coxeter element~$c = (1,2,3,4) \in \Sym_4$ to be the long cycle with reduced word $c = \s_1\s_2\s_3$.
  \Cref{fig:examples} shows cluster variables, $d$- and $g$-vectors and $F$-polynomials for the initial mutation matrix
  \[
  \mutmatrix_\c =
    \begin{pmatrix}
      0        & 1       & 0 \\
      \mii{1}  & 0       & 1 \\
      0        & \mii{1} & 0 \\
      1        & 0       & 0 \\
      0        & 1       & 0 \\
      0        & 0       & 1 
    \end{pmatrix}\ .
  \]
\end{example}

\begin{example}[$B_2$-example]
\label{ex:B2mutmatrix_2}
  Take $W = \Sym_2^B$ the group of signed permutations with simple generators
  \[
    \sref = \big\{ s_1 = (1,2),~ s_2 = (2,\mii{2}) \big\}
  \]
  where $s_1$ is the usual adjacent transposition interchanging the standard basis elements $e_1$ and $e_2$, and where $s_2$ interchanges $e_2$ and $-e_2$.
  $W$ acts on $V = \RR^2$, equipped with the standard inner product.
  We choose
  \[
    \Delta = \big\{ \alpha_1 = 2\mii{2},~ \alpha_2 = 02 \big\}
  \]
  as a basis of $V$.
  With notation as above, we obtain
  \begin{align*}
    \Phiplus &= \big\{ 2\mii{2}\hspace{7px},\ 
    20\hspace{7px},\ 
    22\hspace{7px},\ 
    02\hspace{7px}  \big\} \\
    &= \big\{ 10_\Delta,\ 
    11_\Delta,\ 
    12_\Delta,\ 
    01_\Delta \big\},
  \end{align*}
  and finally $\Phipm = \Phiplus \sqcup -\Delta$ and $\Phi = \Phiplus \sqcup -\Phiplus$.
  In this case, $n = |\sref| = 2$ and $N=|\Phiplus|=4$.
  The corresponding Cartan matrix is
  \[
  C =
    \begin{pmatrix}
      2       & \mii{1} \\
      \mii{2} & 2
    \end{pmatrix}
  \]
  and the fundamental weights are
  \[
    \nabla = \big\{ \omega_1 = 20 = 11_\Delta,\quad \omega_2 = 11 = \tfrac{1}{2}(12_\Delta) \big\}.
  \]
  Fix the Coxeter element~$c = (1,2,\mii 1, \mii 2) \in \Sym_2^B$ to be the long cycle with reduced word $\c = \s_1\s_2$.
  \Cref{fig:examples} shows cluster variables, $d$- and $g$-vectors and $F$-polynomials for the initial mutation matrix
  \begin{equation*}
  \mutmatrix_\c 
  \ = \ 
  \begin{pmatrix}  
  0 & 1 \\
  \mii{2}  & 0 \\ 
  1 & 0 \\
  0 & 1 
  \end{pmatrix}.
  \end{equation*}
  For later reference, \Cref{fig:B2gfan} shows the $g$-vector fan in the weight basis and the Newton polytopes of the $F$-polynomials in the root basis in this case.

  \tikzstyle{poly} = [fill=blue!25, line width=1.2pt, opacity=0.75, draw=blue!40]
  \tikzstyle{segment} = [draw=blue!40, line width=1.2pt, line cap=round]
  \tikzstyle{vertex} = [fill=blue!50, draw=blue!50, line width=1.2pt]
  \tikzstyle{fan} = [fill=red!20, opacity=0.5]
  \tikzstyle{ray} = [draw=red!50, line width=1.4pt, line cap=round]
  \tikzstyle{coords} = [draw=black!60,->,>=latex,line width=0.7pt, line cap=round]
  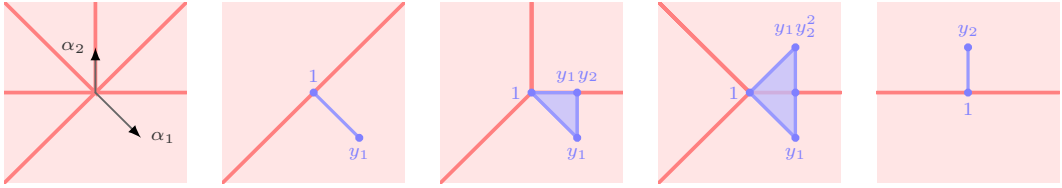
\begin{figure}
  	\centering
  	\begin{tikzpicture}[scale=0.6]
  	\clip (-2,-2) rectangle (2,2);
  	\path[fan] (-2,-2) -- (2,-2) -- (2,2) -- (-2,2) -- cycle; 
  	\path[ray] (2,0) -- (0,0) -- (2,2);
  	\path[ray] (0,2) -- (0,0) -- (-2,2);
  	\path[ray] (-2,0) -- (0,0) -- (-2,-2);    
  	\path[coords] (0,0) -- (1,-1) node[right] {\scriptsize \color{black!80} $\alpha_1$};;
  	\path[coords] (0,0) -- (0,1) node[left] {\scriptsize \color{black!80} $\alpha_2$};
  	\end{tikzpicture}
  	\quad
  	\begin{tikzpicture}[scale=0.6]
  	\clip (-2,-2) rectangle (2,2);
  	\path[fan] (-2,-2) -- (2,-2) -- (2,2) -- (-2,2) -- cycle; 
  	\path[ray] (-2,-2) -- (2,2);
	\path[segment] (0,0) -- (1,-1);
 	\draw[vertex] (1,-1) node[below] {\scriptsize \color{blue!50} $y_1$} circle (1.5pt);
 	\draw[vertex] (0,0) node[above] {\scriptsize \color{blue!50} $1$} circle (1.5pt);
  	\end{tikzpicture}
  	\quad
  	\begin{tikzpicture}[scale=0.6]
  	\clip (-2,-2) rectangle (2,2);
  	\path[fan] (-2,-2) -- (2,-2) -- (2,2) -- (-2,2) -- cycle; 
  	\path[ray] (-2,-2) -- (0,0) -- (0,2) -- (0,0) -- (2,0);
  	\path[poly] (0,0) -- (1,-1) -- (1,0)  - -cycle;
  	\path[segment] (0,0) -- (1,-1) -- (1,0)  - -cycle;
  	\draw[vertex] (1,-1) node[below] {\scriptsize \color{blue!50} $y_1$} circle (1.5pt);
  	\draw[vertex] (1,0) node[above] {\scriptsize \color{blue!50} $y_1y_2$} circle (1.5pt);
  	\draw[vertex] (0,0) node[left] {\scriptsize \color{blue!50} $1$} circle (1.5pt);
  	\end{tikzpicture}
  	\quad
  	\begin{tikzpicture}[scale=0.6]
  	\clip (-2,-2) rectangle (2,2);
  	\path[fan] (-2,-2) -- (2,-2) -- (2,2) -- (-2,2) -- cycle; 
  	\path[ray] (-2,-2) -- (0,0) -- (-2,2) -- (0,0) -- (2,0);
  	\path[poly] (0,0) -- (1,-1) -- (1,1)  - -cycle;
  	\path[segment] (0,0) -- (1,-1) -- (1,1)  - -cycle;
  	\draw[vertex] (1,-1) node[below] {\scriptsize \color{blue!50} $y_1$} circle (1.5pt);
  	\draw[vertex] (1,1) node[above] {\scriptsize \color{blue!50} $y_1y_2^2$} circle (1.5pt);
  	\draw[vertex] (0,0) node[left] {\scriptsize \color{blue!50} $1$} circle (1.5pt);
  	\draw[vertex] (1,0) node[left] {} circle (1.5pt);
  	\end{tikzpicture}
  	\quad
  	\begin{tikzpicture}[scale=0.6]
  	\clip (-2,-2) rectangle (2,2);
  	\path[fan] (-2,-2) -- (2,-2) -- (2,2) -- (-2,2) -- cycle; 
  	\path[ray] (-2,0) -- (2,0);
  	\path[segment] (0,0) -- (0,1);
  	\draw[vertex] (0,1) node[above] {\scriptsize \color{blue!50} $y_2$} circle (1.5pt);
  	\draw[vertex] (0,0) node[below] {\scriptsize \color{blue!50} $1$} circle (1.5pt);
  	\end{tikzpicture}
  	\caption{The $g$-vector fan of type $B_2$ from \Cref{ex:B2mutmatrix_2} (left) is the common refinement of the normal fans of the Newton polytopes  (blue) of the $F$-polynomials. }
  	\label{fig:B2gfan}
  \end{figure}
\end{example}

\bigskip

Let $\wo\in W$ be the unique longest element in weak order.
For a given word $\Q = \q_1\cdots\q_m$ in the simple system $\sref$ define the \Dfn{(spherical) subword complex} $\subwordComplex(\Q)$ as the simplicial complex of (positions of) letters in~$\Q$ whose complement contains a reduced word of~$\wo$. A more general version of these complexes were introduced by A.~Knutson and E.~Miller in~\cite{KM2004}.
By definition, the facets of $\subwordComplex(\Q)$ are subwords of~$\Q$ whose complements are reduced words for~$\wo$.
We consider facets as sorted lists of indices, written in set notation.
Moreover define~$\greedy$ and $\antigreedy$ to be the lexicographically first and last facets, respectively, and call them \Dfn{greedy facet} and \Dfn{antigreedy facet}.
The following notions were introduced and studied for general subword complexes in~\cite{CLS-14, Pilaud-Stump-2015}.
For $\Q = \q_1 \cdots \q_m$ and any facet~$I \in \subwordComplex(\Q)$ associate a \Dfn{root function} $\Root{I}{\cdot} : [m] \to \Phi = W(\Delta) \subseteq V$ and a \Dfn{weight function}~$\Weight{I}{\cdot} : [m] \to W(\nabla) \subseteq V$ defined by
\[
\Root{I}{k} \ = \ \wordprod{\Q}{[k-1] \ssm I}(\alpha_{q_k}) \quad \text{and} \quad \Weight{I}{k} \ = \ \wordprod{\Q}{[k-1] \ssm I}(\omega_{q_k}),
\]
where~$\wordprod{\Q}{X}$ denotes the product of the simple reflections~$q_x \in \Q$, for~$x \in X\subseteq [m]$, in the order given by~$\Q$.
It is well known, see~\cite[Theorem~3.7]{KM2004}, that $\subwordComplex(\Q)$ is a simplicial sphere, thus for a given facet $I$ and index $i\in I$ there exists a unique adjacent facet $J$ with $I\setminus i = J\setminus j$.
We call the transition from $I$ to $J$ the \Dfn{flip} of~$i$ in~$I$ and if $i<j$ such a flip is called \Dfn{increasing}, in which case we write $I\prec J$.
This yields a poset structure on the set of facets of $\subwordComplex(\Q)$ with $\greedy$ as unique minimal element and $\antigreedy$ as unique maximal element. 

\medskip

Following~\cite{CLS-14}, the (abstract) cluster complex $S(\mutmatrix_\c)$ can be seen as a subword complex as follows.
Denote by $\cwo{c}$ the \Dfn{Coxeter-sorting word} (or \Dfn{$\c$-sorting word}) of $\wo$, i.e., the lexicographically first subword of $\c^N$ that is a reduced word for $\wo$.
The notion of Coxeter-sorting words was introduced by N.~Reading in \cite{R2007} and is an essential ingredient in the combinatorial descriptions of finite type cluster algebras and, in particular, in the description of cluster complexes in terms of subword complexes.
In this setting we get the cluster complex as
\begin{equation}
S(\mutmatrix_\c) \cong \clusterComplex. \label{eq:Smutiso}
\end{equation}

\begin{example}[$A_3$-example]
\label{ex:A3mutmatrix}
  For the Coxeter element $c = s_1s_2s_3$ with fixed reduced word $\c = \s_1\s_2\s_3$ we identify the letter~$\s_i$ with its index~$i$.
  The $\c$-sorting word of~$\wo$ then is $\cwo{\c} = 123121$ and we obtain $\c \cwo{c} = 123123121$ for the subword complex $\clusterComplex$.
  The values of the root function are given by
  \setlength{\extrarowheight}{3pt}
  \setlength{\tabcolsep}{5pt}
  \[\resizebox{12cm}{3.75cm}{
    \begin{tabular}{l|| c|c|c| c|c|c| c|c|c}
    \multirow{2}{*}{$I$}
    & \multicolumn{9}{c}{$\Root{I}{\cdot}$}  
    \\
    & 1 & 2 & 3 & 4 & 5 & 6 & 7 & 8 & 9
    \\\hline
    123 $=\greedy$      & 1\mi{1}00 & 01\mi{1}0 & 001\mi{1} 
    & 1\mi{1}00 & 10\mi{1}0 & 100\mi{1} 
    & 01\mi{1}0 & 010\mi{1} & 001\mi{1}
    \\\hline
    129                 & 1\mi{1}00 & 01\mi{1}0 & 001\mi{1}
    & 1\mi{1}00 & 100\mi{1} & 10\mi{1}0
    & 010\mi{1} & 01\mi{1}0 & 00\mi{1}1
    \\\hline
    137                 & 1\mi{1}00 & 01\mi{1}0 & 010\mi{1}
    & 10\mi{1}0 & 1\mi{1}00 & 100\mi{1}
    & 0\mi{1}10 & 010\mi{1} & 001\mi{1}
    \\\hline
    178                 & 1\mi{1}00 & 01\mi{1}0 & 010\mi{1}
    & 10\mi{1}0 & 100\mi{1} & 1\mi{1}00
    & 001\mi{1} & 0\mi{1}01 & 001\mi{1}
    \\\hline
    189                 & 1\mi{1}00 & 01\mi{1}0 & 010\mi{1}
    & 10\mi{1}0 & 100\mi{1} & 1\mi{1}00
    & 001\mi{1} & 01\mi{1}0 & 001\mi{1}
    \\\hline
    234                 & 1\mi{1}00 & 10\mi{1}0 & 001\mi{1}
    & \mi{1}100 & 10\mi{1}0 & 100\mi{1}
    & 01\mi{1}0 & 010\mi{1} & 001\mi{1}
    \\\hline
    249                 & 1\mi{1}00 & 10\mi{1}0 & 001\mi{1}
    & \mi{1}100 & 100\mi{1} & 10\mi{1}0
    & 010\mi{1} & 01\mi{1}0 & 001\mi{1}
    \\\hline
    345                 & 1\mi{1}00 & 10\mi{1}0 & 100\mi{1}
    & 01\mi{1}0 & \mi{1}010 & 100\mi{1}
    & 01\mi{1}0 & 010\mi{1} & 001\mi{1}
    \\\hline
    357                 & 1\mi{1}00 & 10\mi{1}0 & 100\mi{1}
    & 01\mi{1}0 & \mi{1}100 & 100\mi{1}
    & 0\mi{1}10 & 010\mi{1} & 001\mi{1}
    \\\hline
    456                 & 1\mi{1}00 & 10\mi{1}0 & 100\mi{1}
    & 01\mi{1}0 & 001\mi{1} & \mi{1}001
    & 01\mi{1}0 & 010\mi{1} & 001\mi{1}
    \\\hline
    469                 & 1\mi{1}00 & 10\mi{1}0 & 100\mi{1}
    & 01\mi{1}0 & 001\mi{1} & \mi{1}010
    & 010\mi{1} & 01\mi{1}0 & 00\mi{1}1
    \\\hline
    567                 & 1\mi{1}00 & 10\mi{1}0 & 100\mi{1}
    & 01\mi{1}0 & 010\mi{1} & \mi{1}001
    & 0\mi{1}10 & 010\mi{1} & 001\mi{1}
    \\\hline
    678                 & 1\mi{1}00 & 10\mi{1}0 & 100\mi{1}
    & 01\mi{1}0 & 010\mi{1} & \mi{1}100
    & 001\mi{1} & 0\mi{1}01 & 001\mi{1}
    \\\hline
    689 $= \antigreedy$ & 1\mi{1}00 & 10\mi{1}0 & 100\mi{1}
    & 01\mi{1}0 & 010\mi{1} & \mi{1}100
    & 001\mi{1} & 0\mi{1}10 & 00\mi{1}1
    \end{tabular}
  }\]
  and the values of the weight function are given by
  \[\resizebox{12cm}{3.75cm}{
    \begin{tabular}{l|| c|c|c| c|c|c| c|c|c}
    \multirow{2}{*}{$I$}
    & \multicolumn{9}{c}{$\Weight{I}{\cdot}$}  
    \\
    & 1 & 2 & 3 & 4 & 5 & 6 & 7 & 8 & 9
    \\\hline
    123 $= \greedy$     & 1000 & 1100 & 1110 & 1000 & 1100 & 1110 & 0100 & 0110 & 0010 
    \\\hline
    129                 & 1000 & 1100 & 1110 & 1000 & 1100 & 1101 & 0100 & 0101 & 0001 
    \\\hline
    137                 & 1000 & 1100 & 1110 & 1000 & 1010 & 1110 & 0010 & 0110 & 0010 
    \\\hline
    178                 & 1000 & 1100 & 1110 & 1000 & 1010 & 1011 & 0010 & 0011 & 0010 
    \\\hline
    189                 & 1000 & 1100 & 1110 & 1000 & 1010 & 1011 & 0010 & 0011 & 0001 
    \\\hline
    234                 & 1000 & 1100 & 1110 & 0100 & 1100 & 1110 & 0100 & 0110 & 0010 
    \\\hline
    249                 & 1000 & 1100 & 1110 & 0100 & 1100 & 1101 & 0100 & 0101 & 0001 
    \\\hline
    345                 & 1000 & 1100 & 1110 & 0100 & 0110 & 1110 & 0100 & 0110 & 0010 
    \\\hline
    357                 & 1000 & 1100 & 1110 & 0100 & 0110 & 1110 & 0010 & 0110 & 0010 
    \\\hline
    456                 & 1000 & 1100 & 1110 & 0100 & 0110 & 0111 & 0100 & 0110 & 0010 
    \\\hline
    469                 & 1000 & 1100 & 1110 & 0100 & 0110 & 0111 & 0100 & 0101 & 0001 
    \\\hline
    567                 & 1000 & 1100 & 1110 & 0100 & 0110 & 0111 & 0010 & 0110 & 0010 
    \\\hline
    678                 & 1000 & 1100 & 1110 & 0100 & 0110 & 0111 & 0010 & 0011 & 0010 
    \\\hline
    689 $= \antigreedy$ & 1000 & 1100 & 1110 & 0100 & 0110 & 0111 & 0010 & 0011 & 0001 
    \end{tabular}
  }\]
\end{example}

\begin{example}[$B_2$-example]
  \label{ex:B2mutmatrix}
  For the Coxeter element $c = s_1s_2$ with reduced word $\c = \s_1\s_2$ we identify the letters $s_i$ with its index~$i$. The $\c$-sorting word of $\wo$ then is $\cwo{c} = 1212$ and we obtain 
  $\c \cwo{c} = 121212$ for the subword complex $\clusterComplex$. 
  The values of the root and weight function are given by
  \setlength{\extrarowheight}{3pt}
  \setlength{\tabcolsep}{6pt}
  \[
  \begin{tabular}{l|| c|c| c|c| c|c|| c|c| c|c| c|c}
  \multirow{2}{*}{$I$}
  & \multicolumn{6}{c||}{$\Root{I}{\cdot}$} 
  & \multicolumn{6}{c}{$\Weight{I}{\cdot}$} 
  \\
  & 1 & 2 & 3 & 4 & 5 & 6 
  & 1 & 2 & 3 & 4 & 5 & 6
  \\\hline
  12 $= \greedy$      & 2\mi{2} & 02      & 2\mi{2} & 20      & 22            & 02
  & 20      & 11      & 20      & 11      & 02            & \mi{1}1
  \\\hline
  16                  & 2\mi{2} & 02      & 22      & 20      & 2\mi{2}       & 0\mi{2}
  & 20      & 11      & 20      & 1\mi{1} & 0\mi{2}       & \mi{1}\mi{1}
  \\\hline
  23                  & 2\mi{2} & 20      & \mi{2}2 & 20      & 22            & 02
  & 20      & 11      & 02      & 11      & 02            & \mi{1}1
  \\\hline
  34                  & 2\mi{2} & 20      & 22      & \mi{2}0 & 22            & 02
  & 20      & 11      & 02      & \mi{1}1 & 02            & \mi{1}1
  \\\hline
  45                  & 2\mi{2} & 20      & 22      & 02      & \mi{2}\mi{2}  & 02
  & 20      & 11      & 02      & \mi{1}1 & \mi{2}0       & \mi{1}1
  \\\hline
  56 $= \antigreedy$  & 2\mi{2} & 20      & 22      & 02      & \mi{2}2       & 0\mi{2}
  & 20      & 11      & 02      & \mi{1}1 & \mi{2}0       & \mi{1}\mi{1}
  \end{tabular}
  \]
\end{example}

It was developed in~\cite{Pilaud-Stump-2015} how one may obtain a generalized associahedron using subword complexes and brick polytopes.
Define the \Dfn{brick vector} of the facet~$I$ of $\clusterComplex$ as
\begin{align}
\brickVector(I) \ = \ \sum\limits_{k = 1}^{N} \big{(} \Weight{I}{n+k} - \Weight{\antigreedy}{n+k} \big{)} \in V, \label{eq:brickvector}
\end{align}
and the \Dfn{brick polytope} $\Assoc$ in~$V$ as the convex hull of all brick vectors of $\clusterComplex$, that is, 
\[
\Assoc \ = \ \conv \bigset{\brickVector(I)}{I \text{ facet of } \clusterComplex}.
\]
It was shown in \cite[Corollary 6.10]{Pilaud-Stump-2015} that $\Assoc$ is a generalized associahedron.

\medskip

As explained in~\cite{BS-2018}, we consider $g$-vectors to live in the weight space.
This is, we embed a $g$-vector $(g_1,\dots,g_n)$ into the vector space~$V$ as $g_1 \omega_1 +\dots + g_n \omega_n \in V$.
With this convention, we have the following previously known proposition.

\begin{proposition}
  The normal fan of $\Assoc$ is the $g$-vector fan.
  This is,
  \[
    \Assoc \in \bigTC{\gFan(\mutmatrix_\c)}.
  \]
\end{proposition}

\begin{proof}
  It is shown in \cite[Proposition 6.6]{Pilaud-Stump-2015} that the facet normals of all facets of $\Assoc$ containing a given brick vector $\brickVector(I)$ for some facet~$I$ of $\clusterComplex$ are $\set{\Weight{I}{i}}{i \in I}$.
  With the above embedding of the $g$-vectors into~$V$, it was then shown in \cite[Corollary 2.10]{BS-2018} that this set coincides with the set of $g$-vectors inside the cluster of $\Alg(\mutmatrix_\c)$ corresponding to~$I$ inside $\clusterComplex$ under the isomorphism in~\eqref{eq:Smutiso} which is also explained in more detail in \Cref{rem:Smutiso}.
\end{proof}

The given definition of the brick polytope differs from the definition given in~\cite{Pilaud-Stump-2015} by a translation and is chosen so that the brick vector $\brickVector(\antigreedy)$ of the antigreedy facet is the origin.
This translation corresponds to the shifted weight function as used in~\cite[Conjecture~2.12]{BS-2018}.
Furthermore, we have for any facet~$I$ of $\clusterComplex$ that $\Weight{I}{k} = \Weight{\antigreedy}{k}$ for all $1 \leq k \leq n$.
This clarifies why we do not consider the first~$n$ weight vectors in the summation in~\eqref{eq:brickvector}.

\medskip

The root function of the greedy facet provides a bijection between the set of positive roots and the positions $n+1,\ldots,n+N$. That is, $\set{\Root{\greedy}{n+k}}{1\leq k \leq N} = \Phiplus$.
As observed in~\cite[Lemma~3.7]{BS-2018}, we moreover have $\Root{\greedy}{n+k} = \Weight{\greedy}{n+k} - \Weight{\antigreedy}{n+k}$ for all $1 \leq k \leq N$.
For $\beta=\Root{\greedy}{n+k} \in\Phiplus$ and a facet~$I$, we sometimes write $\Weight{I}{\beta} := \Weight{I}{n+k}$ for simplicity, and define
\[
\Assocbeta = \conv\bigset{\Weight{I}{\beta} - \Weight{\antigreedy}{\beta}}{I \text{ facet of } \clusterComplex}\ .
\]

\begin{remark}
\label{rem:Smutiso}
  This identification of the positions $n+1,\dots,n+N$ and $\Phiplus$ is the same as the isomorphism in~\eqref{eq:Smutiso} in the following sense.
  As known since~\cite{FZ2003}, sending a cluster variable $u_\beta(\x,\y)$ to its $d$-vector~$\beta$ is a bijection between cluster variables and almost positive roots $\Phipm$.
  Identifying the positions $1,\dots,n$ with the simple negative roots $-\alpha_1,\dots,-\alpha_n$ in this order and the above identification between positions $n+1,\dots,n+N$ and $\Phiplus$ is a bijection between cluster variables and positions $1,\dots,n,n+1,\dots,n+N$ and this bijection induces the bijection used in~\eqref{eq:Smutiso}.
  In particular, the polytope $\Assocbeta$ naturally correspond to the cluster variable~$u_\beta$.
  This correspondence turns out to be a structural correspondence as discussed in \Cref{sec:fpolys} where we show that $\Assocbeta = \Newton{F_\beta} = \Newton{u_\beta(\onevar,\y)}$ is the Newton polytope of the $F$-polynomial associated to this cluster variable.
\end{remark}

\begin{example}[$A_3$-example]\label{ex:shifted_weights_a3}
  We display the shifted weight function for positions $n+1, \dots, n+N$ and the brick 
  vector in the following shifted weight table:
  \setlength{\extrarowheight}{3pt}
  \setlength{\tabcolsep}{3pt}
  \[\resizebox{\textwidth}{4cm}{
    \begin{tabular}{l|| c|c| >{\columncolor{gray!20}}c| 
      c|c|c|| c}
    \multirow{2}{*}{$I$}
    & \multicolumn{6}{c||}{$\Weight{I}{\cdot} - \Weight{\antigreedy}{\cdot}$}
    & \multirow{2}{*}{$\brickVector(I)$}
    \\
    ~ & 4 & 5 & 6 & 7 & 8 & 9 & ~
    \\\hline
    123                   & 1\mi{1}00       = \D{100} 
    & 10\mi{1}0       = \D{110}
    & 100\mi{1}       = \D{111} 
    & 01\mi{1}0       = \D{010}
    & 010\mi{1}       = \D{011} 
    & 001\mi{1}       = \D{001} 
    & 31\mi{1}\mi{3}  = \D{343}
    \\\hline
    129                   & 1\mi{1}00       = \D{100} 
    & 10\mi{1}0       = \D{110}
    & 10\mi{1}0       = \D{110}
    & 01\mi{1}0       = \D{010}
    & 01\mi{1}0       = \D{010}
    & 0000            = \D{000}
    & 31\mi{4}0       = \D{340}
    \\\hline
    137                   & 1\mi{1}00       = \D{100}
    & 1\mi{1}00       = \D{100}
    & 100\mi{1}       = \D{111}
    & 0000            = \D{000}
    & 010\mi{1}       = \D{011}
    & 001\mi{1}       = \D{001}
    & 3\mi{1}1\mi{3}  = \D{323}
    \\\hline
    178                   & 1\mi{1}00       = \D{100}
    & 1\mi{1}00       = \D{100}
    & 1\mi{1}00       = \D{100}
    & 0000            = \D{000}
    & 0000            = \D{000}
    & 001\mi{1}       = \D{001}
    & 3\mi{3}1\mi{1}  = \D{301}
    \\\hline
    189                   & 1\mi{1}00       = \D{100}
    & 1\mi{1}00       = \D{100}
    & 1\mi{1}00       = \D{100}
    & 0000            = \D{000}
    & 0000            = \D{000}
    & 0000            = \D{000}
    & 3\mi{3}00       = \D{300}
    \\\hline
    234                   & 0000            = \D{000}
    & 10\mi{1}0       = \D{110}
    & 100\mi{1}       = \D{111}
    & 01\mi{1}0       = \D{010}
    & 010\mi{1}       = \D{011}
    & 001\mi{1}       = \D{001}
    & 22\mi{1}\mi{3}  = \D{243}
    \\\hline
    249                   & 0000            = \D{000}
    & 10\mi{1}0       = \D{110}
    & 10\mi{1}0       = \D{110}
    & 01\mi{1}0       = \D{010}
    & 01\mi{1}0       = \D{010}
    & 0000            = \D{000}
    & 22\mi{4}0       = \D{240}
    \\\hline
    345                   & 0000            = \D{000}
    & 0000            = \D{000}
    & 100\mi{1}       = \D{111}
    & 01\mi{1}0       = \D{010}
    & 010\mi{1}       = \D{011}
    & 001\mi{1}       = \D{001}
    & 120\mi{3}       = \D{133}
    \\\hline
    357                   & 0000            = \D{000}
    & 0000            = \D{000}
    & 100\mi{1}       = \D{111}
    & 0000            = \D{000}
    & 010\mi{1}       = \D{011}
    & 001\mi{1}       = \D{001}
    & 111\mi{3}       = \D{123}
    \\\hline
    456                   & 0000            = \D{000}
    & 0000            = \D{000}
    & 0000            = \D{000}
    & 01\mi{1}0       = \D{010}
    & 010\mi{1}       = \D{011}
    & 001\mi{1}       = \D{001}
    & 020\mi{2}       = \D{022}
    \\\hline
    469                   & 0000            = \D{000}
    & 0000            = \D{000}
    & 0000            = \D{000}
    & 01\mi{1}0       = \D{010}
    & 01\mi{1}0       = \D{010}
    & 0000            = \D{000}
    & 02\mi{2}0       = \D{020}
    \\\hline
    567                   & 0000            = \D{000}
    & 0000            = \D{000}
    & 0000            = \D{000}
    & 0000            = \D{000}
    & 010\mi{1}       = \D{011}
    & 001\mi{1}       = \D{001}
    & 011\mi{2}       = \D{012}
    \\\hline
    678                   & 0000            = \D{000}
    & 0000            = \D{000}
    & 0000            = \D{000}
    & 0000            = \D{000}
    & 0000            = \D{000}
    & 001\mi{1}       = \D{001}
    & 001\mi{1}       = \D{001}
    \\\hline
    689                   & 0000            = \D{000}
    & 0000            = \D{000}
    & 0000            = \D{000}
    & 0000            = \D{000}
    & 0000            = \D{000}
    & 0000            = \D{000}
    & 0000            = \D{000}
    \end{tabular}
  }\]
  Especially, the gray marked column corresponds to the polytope 
  \[
  \Asso_{111_\Delta}(\mutmatrix_\c) = \conv
  \big\{ 111_\Delta,~ 110_\Delta,~ 100_\Delta,~ 000_\Delta \big\}.
  \]
\end{example}

\begin{example}[$B_2$-example]\label{ex:shifted_weights_b2}
  We display the shifted weight function for positions $n+1, \dots, n+N$ and the brick 
  vector in the following shifted weight table:
  \setlength{\extrarowheight}{3pt}
  \setlength{\tabcolsep}{6pt}
  \[
  \begin{tabular}{l|| c|c| >{\columncolor{gray!20}}
    c|c|| c}
  \multirow{2}{*}{$I$}
  & \multicolumn{4}{c||}{$\Weight{I}{\cdot} - \Weight{\antigreedy}{\cdot}$}
  & \multirow{2}{*}{$\brickVector(I)$}
  \\
  ~ & 3 & 4 & 5 & 6 & ~
  \\\hline
  12                  & 2\mi{2} = \D{10}        & 20      = \D{11} 
  & 22      = \D{12}        & 02      = \D{01}
  & 62      = \D{34}
  \\\hline
  16                  & 2\mi{2} = \D{10}        & 2\mi{2} = \D{10}  
  & 2\mi{2} = \D{10}        & 00      = \D{00}
  & 6\mi{6} = \D{30}
  \\\hline
  23                  & 00      = \D{00}        & 20      = \D{11} 
  & 22      = \D{12}        & 02      = \D{01}
  & 44      = \D{24}
  \\\hline
  34                  & 00      = \D{00}        & 00      = \D{00} 
  & 22      = \D{12}        & 02      = \D{01}
  & 24      = \D{13}
  \\\hline
  45                  & 00      = \D{00}        & 00      = \D{00} 
  & 00      = \D{00}        & 02      = \D{01}
  & 02      = \D{01}
  \\\hline
  56                  & 00      = \D{00}        & 00      = \D{00} 
  & 00      = \D{00}        & 00      = \D{00} 
  & 00      = \D{00}
  \end{tabular}
  \]
  The gray marked column corresponds to the polytope 
  \[
  \Asso_{12_\Delta}(\mutmatrix_\c) = \conv
  \big\{ 12_\Delta,~ 10_\Delta,~ 00_\Delta \big\},
  \]
  and the brick polytope $\Assoc$ can be seen in \Cref{fig:B2brickpoly}.
  
  \tikzstyle{poly} = [fill=blue!25, line width=1.2pt, opacity=0.75, draw=blue!40]
  \tikzstyle{segment} = [draw=blue!40, line width=1.2pt, line cap=round]
  \tikzstyle{vertex} = [fill=blue!50, draw=blue!50, line width=1.2pt]
  \tikzstyle{fan} = [fill=red!20, opacity=0.5]
  \tikzstyle{ray} = [draw=red!50, line width=1.4pt, line cap=round]
  \tikzstyle{coords} = [draw=black!60,->,>=latex,line width=0.7pt, line cap=round]
  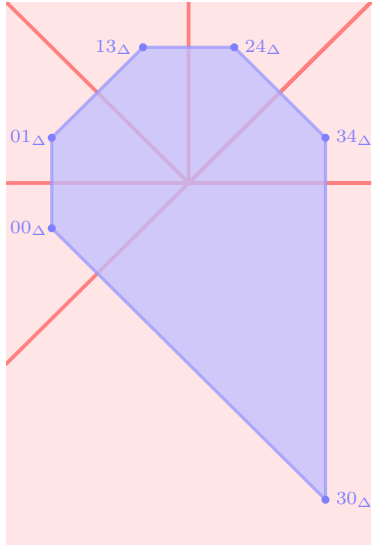
\begin{figure}
    \centering
    \begin{tikzpicture}[scale=0.6]
    \clip (-1,-7) rectangle (7,5);
    \path[fan] (-1,-7) -- (-1,5) -- (7,5) -- (7,-7) -- cycle; 
    \path[ray] (-1,1) -- (7,1) -- (3,1) -- (3,7) -- (-1,5) -- (3,1) -- (7,5) -- (-1,-3);
    \path[poly] (0,0) -- (0,2) -- (2,4) -- (4,4) -- (6,2) -- (6,-6)  - -cycle;
    
    \draw[vertex] (0,0)  circle (1.5pt);
    \node at (-0.5,0) {\scriptsize \color{blue!50} $00_\Delta$};
    \draw[vertex] (0,2)  circle (1.5pt);
    \node at (-0.5,2) {\scriptsize \color{blue!50} $01_\Delta$};
    \draw[vertex] (2,4) node[left] {\scriptsize \color{blue!50} $13_\Delta$} circle (1.5pt);
    \draw[vertex] (4,4) node[right] {\scriptsize \color{blue!50} $24_\Delta$} circle (1.5pt);
    \draw[vertex] (6,2) node[right] {\scriptsize \color{blue!50} $34_\Delta$} circle (1.5pt);
    \draw[vertex] (6,-6) node[right] {\scriptsize \color{blue!50} $30_\Delta$} circle (1.5pt);

    \end{tikzpicture}
    
    \caption{The brick polytope $\Assoc$ of type $B_2$ from \Cref{ex:shifted_weights_b2} and its outer normal fan, centered at $\tfrac{1}{2}(34_\Delta)$. }
    \label{fig:B2brickpoly}
  \end{figure}
\end{example}

We next collect several properties of the polytopes $\Assocbeta$.
We first recall the following crucial lemma.

\begin{lemma}[{\cite[Lemmas 4.4~\&~4.5]{Pilaud-Stump-2015}}]
  \label{lem:weights_under_flips}
  Let $I\setminus i = J\setminus j$ with $i<j$ be two facets of $\clusterComplex$. For any $k\in \{1,\hdots,n+N\}$ we have
  \[
  \Weight{J}{k} = \Weight{I}{k} - \lambda \Root{I}{i} \text{ for some } \lambda\in\mathbb{Z}_{\geq 0} .
  \]
  Moreover for the brick vectors we obtain 
  \[
  \brickVector(J) = \brickVector(I) - \lambda \Root{I}{i} \text{ for some }   \lambda\in\mathbb{Z}_{> 0}.
  \]
\end{lemma}

For a set $X\subseteq \Phiplus$ of positive roots, we set $\brickVector_X(I)=\sum\limits_{\beta \in X} \big(\Weight{I}{\beta} - \Weight{\antigreedy}{\beta}\big)$ and define the polytope $\Asso_X(\mutmatrix_\c) \subset V$ as
\[
\Asso_X(\mutmatrix_\c) := \conv\bigset{\brickVector_X(I)}{ I \text{ facet of } \clusterComplex\ }.
\]
We state the following mild generalization of \cite[Proposition~5.17]{Pilaud-Stump-2015} for the present context.
The proof given there also applies in the present generality and indeed for all \emph{root independent} subword complexes as briefly defined in \Cref{sec:examples} below.

\begin{proposition}
  \label{cor:verts_of_subsum}
  We have the Minkowski decomposition
  \[
  \Asso_X(\mutmatrix_\c) = \sum\limits_{\beta\in X} \Assocbeta.
  \]
\end{proposition}

\begin{proof}
  We may neglect the contributions of the shifts by $\Weight{\antigreedy}{\cdot}$, as these cancel in all considerations.
  By definition we have
  \[
  \Asso_X(\mutmatrix_\c) \subseteq \sum \limits_{\beta \in X} \Assocbeta.
  \]
  To obtain equality we show that every vertex of $\sum_{\beta \in X} \Assocbeta$ is also a vertex of $\Asso_X(\mutmatrix_\c)$.
  Consider a linear functional $f: V \rightarrow \RR$.
  For two adjacent facets $I \setminus i = J \setminus j$ of $\clusterComplex$ and a positive root $\beta \in X$ we have by \Cref{lem:weights_under_flips} that either
  $f( \Weight{I}{\beta} ) = f( \Weight{J}{\beta} )$ or 
  $f( \Weight{I}{\beta} ) - f( \Weight{J}{\beta} )$ has the same sign as 
  $f( \brickVector_X(I) ) - f( \brickVector_X(J) )$.
  Therefore a facet~$I_f$ maximizes $f(\brickVector_X(\cdot))$ among all facets if and only if it maximizes $f(\Weight{\cdot}{\beta})$ for every $\beta \in X$.
  
  Let now~$v$ be a vertex of the Minkowski sum $\sum_{\beta\in X} \Assocbeta$ and let $f: V \rightarrow \RR$ be a linear functional maximized at~$v$.
  Thus,~$v = \sum_{\beta \in X} v_\beta$ such that $v_\beta$ maximizes~$f$ for $\Assocbeta$.
  
  On the other hand,~$f$ is also maximized by some vertex $\brickVector_X(I_f)$ of $\Asso_X(\mutmatrix_\c)$.
  By the previous consideration,~$f$ thus maximizes $\Weight{I_{f}}{\beta}$ for every $\beta \in X$ and we obtain $v_\beta = \Weight{I_{f}}{\beta}$.
  Hence $v = \sum_{\beta \in X} \Weight{I_{f}}{\beta} = \brickVector_X(I_{f})$.
\end{proof}

The description of the Minkowski decomposition of the brick polytope in the previous proposition also yields the following corollary.

\begin{corollary}
  \label{cor:verts_of_subsum_2}
  The set of vertices of $\Asso_X(\mutmatrix_\c)$ is $\bigset{\brickVector_X(I)}{ I \text{ facet of } \clusterComplex}$.
\end{corollary}

\begin{example}[$A_3$-example]
  For $X = \Phiplus \setminus \{111_\Delta\}$ the polytope $\Asso_X(\mutmatrix_\c)$ is given by
  \[ 
  \begin{tabular}{rl}
  $\Asso_X(\mutmatrix_\c)$ 
  & $= \ \conv \big\{ 21\mii{1}\mii{2},~ 21\mii{3}0,~ 2\mii{1}1\mii{2},~ 
  2\mii{2}1\mii{1},~ 2\mii{2}00,~ 12\mii{1}\mii{2},~
  12\mii{3}0,~$                                       \\[5pt]
  & $\hspace{15mm}  020\mii{2},~ 011\mii{2},~ 020\mii{2},~
  02\mii{2}0,~ 011\mii{2},~ 001\mii{1},~ 0000 \big\}$ \\[10pt]
  
  & $= \ \conv \big\{ 232_\Delta,\hspace{2px}
  230_\Delta,\hspace{2px}
  212_\Delta,\hspace{2px}
  201_\Delta,\hspace{2px}
  200_\Delta,\hspace{2px}
  132_\Delta,\hspace{2px}
  130_\Delta,\hspace{2px}$\\[5pt]
  & $\hspace{15mm}    022_\Delta,\hspace{2px} 
  012_\Delta,\hspace{2px} 
  022_\Delta,\hspace{2px}
  020_\Delta,\hspace{2px}
  012_\Delta,\hspace{2px}
  001_\Delta,\hspace{2px}
  000_\Delta \big\}$.
  \end{tabular}
  \]
  For later reference we note that $ \Root{\greedy}{6} = 111_\Delta$ and
  \[
  \brickVector_X( \{ 345 \} ) = \brickVector_X( \{ 456 \}) = 022_\Delta, \quad \brickVector_X( \{ 357 \})  = \brickVector_X( \{ 567 \}) = 012_\Delta\ .
  \]
\end{example}

\begin{example}[$B_2$-example]
  For $X = \Phiplus \setminus \{12_\Delta\}$ the polytope $\Asso_X(\mutmatrix_\c)$ is given by
  \[ 
  \begin{tabular}{rl}
  $\Asso_X(\mutmatrix_\c)$ 
  & $= \ \conv \big\{ 40,\hspace{10px}
  4\mii{4},\hspace{10px}
  22,\hspace{10px}
  02,\hspace{10px}
  02,\hspace{11px}
  00 \hspace{8pt}\big\}$ \\[5pt]
  & $= \ \conv \big\{ 22_\Delta,~ 20_\Delta,~ 12_\Delta,~ 
  01_\Delta,~ 01_\Delta,~ 00_\Delta \big\}$.
  \end{tabular}
  \]
  For later reference we note that $\brickVector_X( \{ 34 \} ) = \brickVector_X( \{ 45 \} )  = 01_\Delta$ and $12_\Delta = \Root{\greedy}{5}$.
\end{example}

We next introduce the following canonical long flip sequence in the subword complex $\clusterComplex$ from the greedy to the antigreedy facet,
\[
\greedy = I_0 \prec I_1 \prec \dots \prec I_N = \antigreedy
\]
where $I_{\ell+1}$ is obtained from $I_\ell$ by flipping the unique index~$i$ in~$I_\ell$ such that $I_{\ell+1} \setminus \{\ell+1+n\} = I_\ell \setminus\{i\}$.
Indeed, up to commutation of consecutive commuting letters, the index~$i$ is the smallest index that yields an increasing flip.
Indeed, there is some flexibility in defining this sequence---any sequence of flips corresponding to \emph{source mutations} in the associated cluster algebra would work.

\begin{example}[$A_3$-example]
  For $\c \cwo{c} = 123123121$ the canonical long flip sequence is given by
  \[
  \greedy = \{1, 2, 3\} \prec
  \{2, 3, 4\} \prec
  \{3, 4, 5\} \prec
  \{4, 5, 6\} \prec
  \{5, 6, 7\} \prec
  \{6, 7, 8\} \prec
  \{6, 8, 9\} = \antigreedy\ .
  \]
\end{example}

\begin{example}[$B_2$-example]
  For $\c \cwo{c} = 121212$ the canonical long flip sequence is given by
  \[
  \greedy = \{1, 2\} \prec
  \{2, 3\} \prec
  \{3, 4\} \prec
  \{4, 5\} \prec
  \{5, 6\} = \antigreedy\ .
  \]
\end{example}

This flip sequence already appeared in \cite[Proposition~6.7]{Pilaud-Stump-2015} and in the proof of \cite[Lemma~3.7]{BS-2018}, where in particular the following property was used.

\begin{lemma}
  \label{lem:difference_weight_sequence}
  For every index $j \in \{ n+1, \ldots , n+N \}$ there exists a unique pair $I_\ell \prec I_{\ell+1}$ in the canonical long flip sequence and an index~$i$ such that $I_\ell \setminus i = I_{\ell+1} \setminus j$.
  Moreover, in this case the weight function $\Weight{I_{\ell+1}}{\cdot}$ is obtained from $\Weight{I_\ell}{\cdot}$ by
  \begin{equation*}
  \Weight{I_{\ell+1}}{k} =
  \begin{cases}
  \Weight{I_\ell}{k} - \Root{I_\ell}{i} & \text{if } k = j , \\
  \Weight{I_\ell}{k}                    & \text{otherwise} . 
  \end{cases}
  \end{equation*}
  In particular, $\Weight{I_{\ell}}{\cdot}$ and $\Weight{I_{\ell+1}}{\cdot}$ only differ for the index~$j$.
\end{lemma}

\begin{proof}
  Up to commutations of consecutive commuting letters in the word $\cw{c} = \q_1\q_2\dots\q_{n+N}$, the facet $I_{\ell}$ consists of the letters $\q_{\ell+1}\dots\q_{\ell+n}$.
  Indeed, we may assume without loss of generality that for each $0 \leq \ell < N$ we have $I_{\ell+1}\setminus I_\ell = \{\ell+n+1\}$.
  Moreover, $\{q_{\ell+1},\dots,q_{\ell+n}\} = \sref$ (this follows, for example, from~\cite[Theorem~2.7]{CLS-14}) and the facets $I_\ell \prec I_{\ell+1}$ may be visualized inside the word $\cw{c}$ as
  \begin{align*}
  I_\ell     &= \q_1\dots\q_{\ell}\ \widehat\q_{\ell+1}\widehat\q_{\ell+2}\dots\widehat\q_{\ell+n}\ \q_{\ell+n+1}\ \q_{\ell+n+2}\ \dots\q_{n+N} \\[5pt]
  I_{\ell+1} &= \q_1\dots\q_{\ell}\         \q_{\ell+1}\widehat\q_{\ell+2}\dots\widehat\q_{\ell+n}\ \widehat\q_{\ell+n+1}\ \q_{\ell+n+2}\ \dots\q_{n+N}
  \end{align*}
  where the letters with hats are omitted and where we assumed, again without loss of generality, that $q_{\ell+1} = q_{\ell+n+1}$.
  The statement of the lemma now follows with
  \[
  \Root{I_\ell}{\ell+1} = \Root{I_\ell}{\ell+n+1} = \Root{I_{\ell+1}}{\ell+1} = -\Root{I_{\ell+1}}{\ell+n+1} = q_1\dots q_\ell(s_{q_{\ell+1}}). \qedhere
  \]
\end{proof}

This lemma yields an interesting combinatorial property of the polytopes $\Assocbeta$ that we do not use further below.

\begin{corollary}
  For every $\beta \in \Phiplus$ the segment connecting $\mathbf{0}$ and 
  $\beta$ is an edge of $\Assocbeta$.
\end{corollary}

\begin{proof}
  As the brick polytope $\Assoc$ realizes $\clusterComplex$ its edges are in one-to-one 
  correspondence to flips in $\clusterComplex$. 
  Combining \Cref{lem:weights_under_flips} and \Cref{cor:verts_of_subsum}
  we obtain a similar result for $\Assocbeta$ saying its edges are in one-to-one
  correspondence with flips that change the weight function $\Weight{\cdot}{\beta}$.
  Applying \Cref{lem:difference_weight_sequence} to the canonical long flip sequence
  \[
  \greedy = I_0 \prec I_1 \prec \hdots \prec I_{N-1} \prec I_N = \antigreedy
  \]
  we obtain for $\beta = \Root{\greedy}{n + i}$ that
  \[\begin{tabular}{lllll}
  $\Weight{ \greedy }{ \beta}$            &
  $\ = \ \Weight{ I_1 }{ \beta}$          &
  $\ = \ \hdots$                          &
  $\ = \ \Weight{ I_{i-1} }{ \beta}$,     &
  and                                     \\
  
  $\Weight{ I_i }{ \beta}$                &
  $\ = \ \hdots$                          &
  $\ = \ \Weight{ I_{N-1} }{ \beta}$      &
  $\ = \ \Weight{ \antigreedy }{ \beta}$. &
  \end{tabular}\]
  As  $\Weight{ \greedy }{ \beta} - \Weight{ \antigreedy }{ \beta} = \beta$ we conclude the statement.
\end{proof}

\subsection{Generators of the type cone}
\label{sec:typecones_polyhedra}

The following definitions mostly follow~\cite{Padrol-Pilaud-2019}.
Let $\fan$ be an essential complete simplicial fan in~$\RR^{d}$.
A \Dfn{polytopal realization} of $\fan$ is a convex polytope in $\RR^{d}$ whose outer normal fan agrees with~$\fan$. 
The space of all polytopal realizations of $\fan$ is called the \Dfn{type cone} of $\fan$, denoted by $\TC{\fan}$, see also~\cite{Mullen-1973}.
A parametrization of $\TC{\fan}$ is commonly described as follows.
Denote by $G \in \RR^{m\times d}$ the matrix whose rows generate the rays of~$\fan$.
Each height vector $h \in \RR^{m}$ defines a polytope
\[
P_h \ = \ \left\{ x \in \RR^{d} \mid Gx \leq h \right\} .
\]
Now the type cone of $\fan$ can be parametrized as the open polyhedral cone
\[
\TC{\fan} \ = \ \left\{ h \in \RR^{m} \mid P_h \text{ has normal fan } \fan \right\} .
\]
We write $P_h \in \TC{\fan}$ by identifying a polytope $P_h$ with its height vector $h \in \RR^{m}$.
With this definition, $\TC{\fan}$ has $d$-dimensional lineality space corresponding to translations in~$\RR^{d}$.
More specifically, for $P_h \in \TC{\fan}$ and a translation vector $b \in \RR^{d}$ we have
\[
P_h + b \ = \ P_{h+Gb} \in \TC{\fan} \enspace ,
\]
Thus the lineality space of $\TC{\fan}$ is given by the image of the matrix~$G$.
We identify $\TC{\fan}$ with its pointed quotient $\TC{\fan} / G\RR^{d}$.
The closure $\TCclosed{\fan}$ is called the \Dfn{closed type cone}. 
The faces of $\TCclosed{\fan}$ correspond to (weak) Minkowki summands of~$P$ with the same normal fan (which are coarsenings of~$\fan$).
In particular, the (extremal) generators of $\TCclosed{P}$ correspond to the indecomposable Minkowski summands of~$P$.

\medskip 

We aim at the description of the type cone $\TC{\gFan(\mutmatrix_\c)}$ of the $g$-vector fan $\gFan(\mutmatrix_\c)$ given in \Cref{thm:indecomp_columns}.
We first state the following lemma which we then use to understand the rays of the type cone.
\begin{lemma}
  \label{lem:simplicial_cone}
  Let $C \subset \RR^m$ be a full-dimensional closed polyhedral cone and let $x = x_1+\cdots+x_m$ for $x_1,\ldots,x_m \in C$ with
  \begin{enumerate}[(i)]
    \item $x$ is an interior point of~$C$ and
    \item $x - x_i \enspace$ is contained in the boundary of~$C$ for every $i \in \{ 1,\ldots,m \}$.
  \end{enumerate}
  Then $C=\cone\{x_1,\ldots,x_m\}$.
  In particular, the cone~$C$ is simplicial.
\end{lemma}
\begin{proof}
  Write $X = \{x_1,\ldots,x_m\}$.
  We first show that $X$ is linearly independent.
  Assuming the contrary, one may express some $x_i$ in terms of $X \setminus \{x_i\}$.
  This would mean that $x = (x - x_i) + x_i$ would be in the linear span of $X \setminus \{x_i\}$.
  By condition~(i), this point is in the interior of~$C$, while it is on the boundary by condition~(ii)---a contradiction.
  It follows that $\cone(X)$ is a simplicial full-dimensional cone inside~$C$.
  As condition~(ii) implies that its boundary is also contained in the boundary of~$C$, we conclude the statement.
\end{proof}

\begin{proof}[Proof of \Cref{thm:indecomp_columns}]
  The $g$-vector fan $\gFan(\mutmatrix_\c)$ is an essential complete simplicial fan in $\RR^{n}$ with $n+N$ rays.
  Therefore, after passing to the quotient by its $n$-dimensional lineality space, the closed type cone $\TCclosed{\gFan(\mutmatrix_\c)}$ is an $N$-dimensional pointed polyhedral cone.
  We aim at applying \Cref{lem:simplicial_cone} using the points $\set{\Assocbeta}{\beta \in \Phiplus}$.
  We have seen in~\eqref{eq:Associntypecone} that
  \[
  \Assoc = \sum\limits_{\beta \in \Phi^+} \Assocbeta
  \]
  is an interior point of $\TCclosed{\gFan(\mutmatrix_\c)}$.
  Therefore, it suffices to show that for each $\gamma \in \Phi^+$ the polytope $\Asso_{\Phiplus\setminus\{\gamma\}}(\mutmatrix_\c)$ is contained in the boundary of $\TCclosed{\gFan(\mutmatrix_\c)}$.
  
  Let $\gamma \in \Phiplus$ and let $j\in\{n+1,\ldots,n+N\}$ be the unique index such that $r(I_g,j)=\gamma$.
  \Cref{lem:difference_weight_sequence} ensures the existence of a unique index~$\ell$ such that~$j$ is contained in $I_{\ell+1}$ but not in~$I_\ell$ in the canonical long flip sequence $I_0 \prec \dots \prec I_N$.
  Since $\Weight{I_{\ell}}{\cdot}$ and $\Weight{I_{\ell+1}}{\cdot}$ only differ for the index~$j$, it follows that
  \[
  \brickVector_{\Phiplus\setminus\{\gamma\}}(I_\ell) = \brickVector_{\Phiplus\setminus\{\gamma\}}(I_{\ell+1}).
  \]
  \Cref{cor:verts_of_subsum} and the second part of \Cref{lem:weights_under_flips} now show that the number of vertices of $\Asso_{\Phiplus\setminus\{\gamma\}}(\mutmatrix_\c)$ is strictly less than the number of vertices of $\Assoc$.
  This means that it is a proper weak Minkowski summand and it is thus not contained in the interior of $\TCclosed{\gFan(\mutmatrix_\c)}$.
  Invoking \Cref{lem:simplicial_cone} yields the proposed statement
  \[
  \TCclosed{\gFan(\mutmatrix_\c)} = \cone\set{\Assocbeta}{\beta \in \Phiplus}
  \]
  and that the type cone is in particular simplicial.
\end{proof}

\subsubsection{Generators of the type cone for general spherical subword complexes}
\label{sec:examples}

We close this section with a brief discussion of properties of type cones for examples of general subword complexes.
It turns out that the situation for cluster complexes is particularly special.
Most importantly, the conclusion of \Cref{lem:difference_weight_sequence} does not hold in general for spherical subword complexes.

\medskip

The complex $\clusterComplex$ is known to have the following properties.
For a word~$\Q$, we call a spherical subword complex $\subwordComplex(\Q)$ \Dfn{root-independent} if the multiset
\[
\Roots{I} = \bigmultiset{ \Root{I}{i} }{ i \in I }
\]
is linearly independent for any (and thus every) facet~$I$
and it is \Dfn{of full support} if every position in~$\Q$ is contained in some facet (meaning that all elements of the ground set are indeed vertices).
Observe that spherical subword complexes of full support are also full-dimensional, meaning that $\Roots{I}$ generates~$V$ for any facet~$I$.
This is an immediate consequence of~\cite[Proposition~3.8]{Pilaud-Stump-2015}.
We conjecture that these properties identify cluster complexes among spherical subword complexes.

\begin{conjecture}
  Let~$Q$ be a word in~$\sref$.
  The following statements are equivalent:
  \begin{enumerate}
    \item Up to commutations of consecutive commuting letters $\Q = \c \cwo{c}$ for some Coxeter element~$c$.
    \item $\subwordComplex(\Q)$ is root-independent and of full support.
  \end{enumerate}
\end{conjecture}

Remark that the first property was shown to be equivalent to the so-called 
SIN-property in \cite[Theorem~2.7]{CLS-14}.
Furthermore they conjecture these subword complexes to maximize the number of facets a subword complex $\subwordComplex(\Q)$ with $|\Q| = n+N$ can have \cite[Conjecture~9.8]{CLS-14}.
\medskip

We next show that relaxing one of the two conditions yields examples for which the conclusion of \Cref{thm:indecomp_columns} does not hold.
We denote by~$P_i$ the polytope corresponding to the $i$-th column in the shifted weight table,
\[
P_i \ = \ \conv \bigset{ \Weight{I}{i} - \Weight{\antigreedy}{i} }
{ I \text{ facet of } \subwordComplex(\Q) }
\]
for the given subword complex $\subwordComplex(\Q)$.

\begin{example}[$B_2$-example]
  For $\Q = 1212121$, the subword complex $\subwordComplex(\Q)$ is of full support but not root-independent as $\Roots{ \greedy } = \Roots{ \{ 1,2,3 \} } = 
  \bigmultibrackets{ 10_\Delta,~ 01_\Delta,~ 10_\Delta }$.
  The list of facets is
  \[
  \big\{ \{1,2,3\},~ \{1,2,6\},~ \{1,3,4\},~ \{1,4,5\},~ \{1,5,6\},~ 
  \{1,6,7\},~ \{2,3,7\},~ \{3,4,7\},~ \{4,5,7\},~ \{5,6,7\} \big\}.
  \]
  One may easily check that its brick polytope is the permutahedron of type~$B_2$ whose normal fan is the Coxeter fan.
  The complete list of polytopes is
  \[
  \begin{tabular}{lll}
  $P_1$ & $= \ P_2$         & $= \ \conv \{ 00_\Delta \}$             \\
  $P_3$ & $= \ P_7$         & $= \ \conv \{ 00_\Delta,~ 10_\Delta \}$ \\
  $P_4$ & ~                 & $= \ \conv \{ 00_\Delta,~ 10_\Delta,~ 
  11_\Delta \}$             \\
  $P_6$ & ~                 & $= \ \conv \{ 00_\Delta,~ 01_\Delta,~ 
  11_\Delta \}$             \\
  $P_5$ & $= \ P_3 + P_{?}$ & $= \ \conv \{ 00_\Delta,~ 10_\Delta,~ 
  12_\Delta,~ 22_\Delta \}$,
  \end{tabular}
  \]
  where $P_{?} = \conv \{ 00_\Delta,~ 12_\Delta \}$ is a missing generator 
  of the type cone.
  Furthermore the sum of~$P_4$ and~$P_6$ can be decomposed into
  \[
  P_4 + P_6 = \conv \{ 00_\Delta,~ 10_\Delta \} +
  \conv \{ 00_\Delta,~ 01_\Delta \} +
  \conv \{ 00_\Delta,~ 11_\Delta \}.
  \]
  In particular, the type cone of the brick polytope is not simplicial.
\end{example}

\begin{example}[$B_2$-example]
  For $\Q = 212212$ the subword complex $\subwordComplex(\Q)$ is root-independent and full-dimensional but not of full support as 
  $\Roots{ \greedy } = \Roots{ \{ 1,3 \} } = 
  \bigmultibrackets{ 10_\Delta,~ 12_\Delta }$ and the list of facets is
  \[
  \big\{ \{1,3\},~ \{1,4\},~ \{3,6\},~ \{4,6\} \big\}.
  \]
  The positions $2$ and $5$ are not contained in any facet. The complete list of 
  polytopes is
  \[
  \begin{tabular}{lll}
  $P_1$ & $= \ P_2$ & $= \ \conv \{ 00_\Delta \}$ \\
  $P_3$ & $= \ P_6$ & $= \ \conv \{ 00_\Delta,~ 01_\Delta \}$ \\
  $P_5$ & $= \ P_3 + P_6$ & $= \ \conv \{ 00_\Delta,~ 02_\Delta \}$ \\
  $P_4$ & $= \ P_3 + P_{?}$ & $= \ \conv \{ 00_\Delta,~ 01_\Delta,~ 
  11_\Delta,~ 12_\Delta \}$,
  \end{tabular}
  \]
  where $P_{?} = \conv \{ 00_\Delta,~ 11_\Delta \}$ is the missing generator 
  of the type cone.
\end{example}

\section{Newton polytopes of \texorpdfstring{$F$}{F}-polynomials from subword complexes}
\label{sec:fpolys}

Let $\Alg(\mutmatrix_\c)$ be the finite type cluster algebra with acyclic initial 
mutation matrix $\mutmatrix_\c$ with principal coefficients and 
denote by $\gFan(\mutmatrix_\c)$ its $g$-vector fan. 
We have seen in~\Cref{thm:indecomp_columns} that the type cone 
$\TC{\gFan(\mutmatrix_\c)}$ is generated by the natural Minkowski summands of the 
brick polytope $\Assoc$, 
\[
\bigTC{ \gFan( \mutmatrix_\c ) } \ = \ 
\cone \bigset{ \Assocbeta }{ \beta \in \Phiplus }.
\]
A description of the generators of $\bigTC{ \gFan( \mutmatrix_\c ) }$ 
was also obtained by combining results from \cite{ABHY-2018, AHHL2020} and 
\cite{Padrol-Pilaud-2019} as follows.
In \cite[Theorem~1]{ABHY-2018} the authors provide polytopal realizations of 
$\gFan( \mutmatrix_\c )$.
This construction produces a generalized associahedron $\mathcal X_p$ for each $p \in \RR_{ > 0 }^{ \Phiplus }$.
It was then shown in~\cite[Theorem~3]{ABHY-2018} (simply-laced types) and in~\cite[Theorem~6.1]{AHHL2020} (multiply-laced types) that $\mathcal X_p$ for $p = e_\beta$ and $\beta \in \Phiplus$ equals the Newton polytope of the $F$-polynomial $F_\beta$.
In \cite[Theorem~2.26]{Padrol-Pilaud-2019}, the authors explain that within the latter constuction $\RR_{> 0}^{\Phiplus}$ can be understood as (a linear transformation of) the type cone $\bigTC{ \gFan( \mutmatrix_\c ) }$.
In particular, this establishes the fact that the Newton polytopes of the $F$-polynomials generate the type cone,
\[
\TC{ \gFan( \mutmatrix_\c ) } \ = \ 
\cone \bigset{ \Newton{F_\beta} }{ \beta \in \Phiplus }.
\]

In order to prove \Cref{cor:NewtonFpoly}, it remains to properly identify which Newton polytope of an $F$-polynomial corresponds to which Minkowski summand of the brick polytope.
This is done using the following property of $F$-polynomials.

\begin{proposition}[\cite{DWZ:2010} (simply-laced types), \cite{Dem10} (multiply-laced types)]
  \label{lem:highest_and_lowest_exponent}
  For every $\beta \in \Phiplus$, the $F$-polynomial $F_\beta = F_\beta(\y)$ has constant term~$1$ and a unique componentwise highest exponent vector given by~$\beta$.
  In particular, $\mathbf{0}$ and $\beta$ are both vertices of $\Newton{F_\beta}$.
\end{proposition}

\begin{proof}
  For simply-laced types, this is \cite[Proposition~3.1 \& Theorem~5.1]{DWZ:2010} and for multiply-laced types, this is \cite[Proposition~9.4]{Dem10}.
\end{proof}

%

This proposition can be rechecked in types~$A_3$ and~$B_2$ in \Cref{fig:examples}.
Now we are ready to proof our second main result.

\begin{proof}[Proof of \Cref{cor:NewtonFpoly}]
  Since $\bigTC{ \gFan( \mutmatrix_\c ) }$ is a simplicial cone of dimension 
  $N=| \Phiplus |$ we already know that the two sets of generators,
  \[
  \bigset{ \Newton{ F_\beta } }{ \beta \in \Phiplus } \ \text{and} \ 
  \bigset{ \Assocbeta }{ \beta \in \Phiplus },
  \]
  are non-redundant and coincide up to scalar factors.
  Let $\beta \in \Phiplus$.
  By \Cref{lem:highest_and_lowest_exponent} the unique maximal and minimal vertices of $\Newton{F_\beta}$ are~$\beta$ and~$\mathbf{0}$, respectively.
  Since $\beta = \brickVector_{\{\beta\}}( \greedy )$ and $\mathbf{0} = \brickVector_{\{\beta\}}( \antigreedy )$, these vectors are by \Cref{cor:verts_of_subsum} vertices of $\Assocbeta$ as well. 
  Applying~\Cref{lem:weights_under_flips} we see that they are the maximal and minimal vertices of $\Assocbeta$, respectively.
  Thus the polytopes $\Newton{ F_\beta }$ and $\Assocbeta$ coincide.
\end{proof}

\section{The tropical positive cluster variety}
\label{sec:tropplus}

In this section, we prove \Cref{thm:tropplusvariety} starting from the type cone description~\eqref{eq:newtonF2} on page~\pageref{eq:newtonF2} in terms of Newton polytopes of $F$-polynomials.
It is independent of the subword complex description and does not make use of it.
We again emphasize that a more general version of \Cref{thm:tropplusvariety} follows from~\cite[Theorems~4.1 \& 4.2]{AHHL2020}.

\medskip

Following~\cite{Speyer-Williams-2005}, we start with the needed notions from tropical geometry.
Let $E \subset \ZZ_{\geq 0}^d$ be non-empty and finite and let $f=\sum_{e \in E} f_e \u^e \in \QQ[\u]$ with $f_e \neq 0$ for all $e \in E$ be a rational polynomial supported on~$E$.
For each weight $\weight \in \RR^d$ we define
\[
E(\weight) \ = \ \underset{e \in E}{\arg\max} (\weight \cdot e)\ .
\]
That is, $E(\weight)$ is the intersection of $E$ with the face of $\Newton{f}=\conv(E)$ that is maximized in direction $\weight$.
The \Dfn{tropical hypersurface} $\Trop(f) \subset \RR^d$ is the collection of those weights $\weight \in \RR^d$ for which $E(\weight)$ contains at least two elements.
$\Trop(f)$ naturally carries the structure of a polyhedral fan,  whose cones are formed by those weights $\weight \in \Trop(f)$ that yield the same~$E(\weight)$.
This fan thus agrees with the codimension-one skeleton of the normal fan of~$\Newton{f}$.

The \Dfn{positive part} $\Tropplus(f)$ of the tropical hypersurface was introduced in~\cite{Speyer-Williams-2005} and is defined as follows.
Split $E=E^+_f \sqcup E^-_f$ according to the signs of the coefficients of~$f$.
That is,
\[
E^+_f = \{ e \in E \mid f_e > 0 \}, \quad E^-_f = \{ e \in E \mid f_e < 0 \}.
\]
Now $\Tropplus(f)$ is defined as the subfan of $\Trop(f)$ consisting of those weights for which neither $E(\weight)\cap E^+_f$ nor $E(\weight)\cap E^-_f$ is empty,
\[
\Tropplus(f) = \bigset{\weight \in \RR^d}{E(\weight)\cap E^+_f \neq \emptyset ~ \text{ and } ~ E(\weight)\cap E^-_f \neq \emptyset}\ .
\]
For any ideal $\ideal \subset \QQ[\u]$ the \Dfn{positive tropical variety} $\Tropplus(\ideal)$ is defined as the intersection of all positive tropical hypersurfaces $\Tropplus(f)$ for $f\in \ideal$.

\medskip

We next move to the positive tropical variety considered here.
Let $\Alg(\mutmatrix)$ be a finite type cluster algebra of rank~$n$ with (not necessarily acyclic) initial mutation matrix~$\mutmatrix$ with principal coefficients.
We denote by $X_\Delta=\{x_1,\ldots,x_n\}$ the set of initial cluster variables and by $X_{\Phiplus}=\set{x_\beta}{\beta \in \Phiplus}$ the set of non-initial cluster variables.
Thus the set of all cluster variables is the disjoint union $X=X_\Delta \sqcup X_{\Phi^+}$.
Furthermore, let $Y=\{y_1,\ldots,y_n\}$ be the set of principle coefficient variables.
Recall that each non-initial cluster variable $x_\beta \in X_{\Phiplus}$ is expressed in terms of the initial seed by
\begin{equation*}
\label{eq:non-init-var}
x_\beta = \frac{p_\beta(\x,\y)}{\x^\beta}
\end{equation*}
where $p_\beta(\x,\y)$ is a subtraction-free polynomial in the initial cluster and coefficient variables and $\x^\beta = x_1^{\beta_1}\dots x_n^{\beta_n}$ for $\beta = (\beta_1,\cdots,\beta_n)_\Delta \in \RR^\Delta$.

Following \cite{Speyer-Williams-2005} we embed $\spec \Alg(\mutmatrix)$ as the affine variety $\variety(\ideal_{\mutmatrix}) \subset \field^{X\sqcup Y}$, where $\ideal_{\mutmatrix}$ is the ideal generated by the non-initial cluster variables, i.e.,
\[
\ideal_{\mutmatrix} \ = \ \left\langle x_\beta \cdot \x^\beta - p_\beta(\x,\y) ~ \big| ~ \beta \in \Phiplus \right\rangle .
\]
Note that in this case the special form of the generators immediately yields a subtraction-free parametrization {$\Psi: (\field^\ast)^{X_\Delta \sqcup Y} \rightarrow \variety(\ideal_{\mutmatrix})\cap (\field^\ast)^{X\sqcup Y}$} given as the graph of the map
\begin{align*}
(\field^\ast)^{X_\Delta \sqcup Y} \ &\longrightarrow \ (\field^\ast)^{X_{\Phi^+}} \enspace , \\
(\x,\y) \ &\longmapsto \ \left( \frac{p_\beta(\x,\y)}{\x^\beta} \right)_{\beta \in \Phiplus}
\end{align*}
We denote by $\Trop \Psi : \RR^{X_\Delta \sqcup Y} \rightarrow \RR^{X \sqcup Y}$ the \Dfn{tropicalization} of the map $\Psi$.
This is the piecewise linear map obtained by replacing every $\times$ in $\Psi$ with a $+$, every $/$ with a $−$ and every $+$ with a $\max$.
The following result is an immediate consequence of~\cite[Proposition~2.6]{Speyer-Williams-2005}.
\begin{proposition}
  \label{cor:param_tropplus}
  The map $\Trop \Psi: \RR^{X_\Delta \sqcup Y} \rightarrow \RR^{X \sqcup Y}$ is a piecewise linear parametrization of the positive tropical variety $\Tropplus(\ideal_{\mutmatrix})$.
\end{proposition}

It should be mentioned here that in \cite{Speyer-Williams-2005} the authors are working over the field of complex Puiseux series, one of the prototype examples of a field with non-trivial valuation.
This is standard in tropical geometry since it reveals strong connections between classical algebraic geometry and tropical geometry.
The ideal $\ideal_{\mutmatrix}$ in \Cref{cor:param_tropplus} is understood over the complex Puiseux series and for the definition of a positive tropical hypersurface for a complex Puiseux polynomial we refer to~\cite{Speyer-Williams-2005}.
However, the map $\Trop \Psi$ stays unchanged when working over~$\QQ$.

\medskip

The domains of linearity of $\Trop \Psi$ form a polyhedral fan in $\RR^{X_\Delta \sqcup Y}$, which we denote by $\fan_{\Psi}$.
Following \cite[Definition~4.2]{Speyer-Williams-2005} we equip $\Tropplus(\ideal_{\mutmatrix})$ with the fan structure obtained by applying $\Trop \Psi$ to $\fan_{\Psi}$.
Whenever we refer to $\Tropplus(\ideal_{\mutmatrix})$ as a polyhedral fan we consider this fan structure.

\begin{proof}[Proof of \Cref{thm:tropplusvariety}]
  By \Cref{cor:param_tropplus} the map $\Trop \Psi: \RR^{X_\Delta \sqcup Y} \rightarrow \Tropplus(\ideal_{\mutmatrix})$ is a piecewise linear parametrization of~$\Tropplus(\ideal_{\mutmatrix})$.
  Moreover, $\Tropplus(\ideal_{\mutmatrix})$ is piecewise linearly isomorphic to $\fan_\Psi$ by construction  .
  For each $\beta \in \Phiplus$ denote by $\fan_\beta$ the normal fan of $\Newton{p_\beta}$.
  The domains of linearity of the coordinate function $\Trop \Psi_\beta(\x,\y) = \Trop (p_\beta(\x,\y) / \x^\beta)$ are the maximal cones of~$\fan_\beta$.
  Thus the domains of linearity of $\Trop \Psi$, and hence the fan structure~$\fan_\Psi$ of the positive tropical variety~$\Tropplus(\ideal_{\mutmatrix})$, are given by the common refinement of these fans~$\fan_\beta$ for $\beta \in \Phiplus$.
  
  \medskip
  
  It follows from \cite[Corollary~6.3]{FZ:CA4} that there exists an affine transformation $T: \RR^{Y} \rightarrow \RR^{X_\Delta \sqcup Y}$ such that for all $\beta \in \Phiplus$ we have
  \[
  \Newton{p_\beta} \ = \ T \left( \Newton{F_\beta} \right)\ .
  \]
  Conversely, $\Newton{F_\beta}$ is obtained from $\Newton{p_\beta}$ by the coordinate projection $\RR^{X_\Delta \sqcup Y} \rightarrow~\RR^{Y}$.
  Therefore each fan $\fan_\beta$ is linearly isomorphic to the normal fan of $\Newton{F_\beta}$.
  By \Cref{cor:NewtonFpoly} the common refinement of these normal fans is the $g$-vector fan~$\gFan$.
  This shows that $\fan_\Psi$ is piecewise linearly isomorphic to $\gFan$.
\end{proof}

\begin{example}[$B_2$-example]
  \label{ex:B2trop}
  We continue \Cref{ex:B2mutmatrix_2}.
  We denote by $X_\Delta=\{x_1,x_2\}$ the initial cluster variables and by $Y=\{y_1,y_2\}$ the principle coefficient variables.
  This yields the non-initial cluster variables
  \begin{align*}
    x_3 = x_{10_\Delta} & = (x_2^2 + y_1)/{x_1} \\[2pt]
    x_4 = x_{11_\Delta} & = (x_1 y_1 y_2 + x_2^2 + y_1)/{x_1 x_2} \\[2pt]
    x_5 = x_{12_\Delta} & = (x_1^2 y_1 y_2^2 + 2 x_1 y_1 y_2 + x_2^2 + y_1)/{x_1 x_2^2} \\[2pt]
    x_6 = x_{01_\Delta} & = (x_1 y_2 + 1)/{x_2}
  \intertext{
  as given in the above example.
  The piecewise linear map $\Trop \Psi : \RR^{X_\Delta \sqcup Y} \rightarrow \RR^{X \sqcup Y}$ has non-trivial coordinate functions}
    \Trop \Psi_{10_\Delta} & = \max(2 x_2 \, , \, y_1) - x_1 \\
    \Trop \Psi_{11_\Delta} & = \max(x_1 + y_1 + y_2 \, , \, 2 x_2 \, , \, y_1) - x_1 - x_2 \\
    \Trop \Psi_{12_\Delta} & = \max(2 x_1 + y_1 + 2 y_2 \, , \, x_1 + y_1 + y_2 \, , \, 2 x_2 \, , \, y_1) - x_1 - 2 x_2 \\
    \Trop \Psi_{01_\Delta} & = \max(x_1 + y_2 \, , \, 0) - x_2 .
  \end{align*}
  The domains of linearity of $\Trop \Psi$ define a complete four-dimensional polyhedral fan $\fan_{\Psi}$ in $\RR^{X_\Delta\sqcup Y} = \RR^4$ with two-dimensional lineality space.
  By intersecting $\fan_{\Psi}$ with the coordinate plane $\RR^{Y}$ by setting $x_1=x_2=0$ we obtain an essential $2$-dimensional fan.
  This fan is the the common coarsening of the normal fans of the $F$-polynomials $F_{10_\Delta}$, $F_{11_\Delta}$, $F_{12_\Delta}$, $F_{01_\Delta}$, see \Cref{fig:B2gfan,,fig:B2trop}. 
  The normal fans are depicted in the dual basis of the root basis, known as the coweight basis, which in this case is given as
  \[
    \nabla^\vee=\{\omega_1^\vee,\ \omega_2^\vee\}=\left\{\tfrac{1}{2}(10),\ \tfrac{1}{2}(11)\right\}.
  \]
\end{example}

\tikzstyle{poly} = [fill=blue!25, line width=1.2pt, opacity=0.75, draw=blue!40]
\tikzstyle{segment} = [draw=blue!40, line width=1.2pt, line cap=round]
\tikzstyle{vertex} = [fill=blue!50, draw=blue!50, line width=1.2pt]
\tikzstyle{fan} = [fill=red!20, opacity=0.5]
\tikzstyle{ray} = [draw=red!50, line width=1.4pt, line cap=round]
\tikzstyle{coords} = [draw=black!60,->,>=latex,line width=0.7pt, line cap=round]
\begin{figure}
 \centering
 \scriptsize
 \begin{tikzpicture}[scale=0.6]
   \clip (-2,-2) rectangle (2,2);
   \path[fan] (-2,-2) -- (2,-2) -- (2,2) -- (-2,2) -- cycle; 
   \path[ray] (2,0) -- (0,0) -- (2,2);
   \path[ray] (0,2) -- (0,0) -- (-2,2);
   \path[ray] (-2,0) -- (0,0) -- (-2,-2);
   \path[coords] (0,0) -- (1.6,0) node[above left] {\scriptsize \color{black!80} $\omega_1^\vee$};;
   \path[coords] (0,0) -- (1.6,1.6) node[left] {\scriptsize \color{black!80} $\omega_2^\vee$};    
 \end{tikzpicture}
 \quad
 \begin{tikzpicture}[scale=0.6]
    \clip (-2,-2) rectangle (2,2);
    \path[fan] (-2,-2) -- (2,-2) -- (2,2) -- (-2,2) -- cycle; 
    \path[ray] (-2,-2) -- (2,2);
    \node at (0.8,-0.8) {$y_1$};
    \node at (-0.8,0.8) {$0$};
 \end{tikzpicture}
 \quad
 \begin{tikzpicture}[scale=0.6]
    \clip (-2,-2) rectangle (2,2);
    \path[fan] (-2,-2) -- (2,-2) -- (2,2) -- (-2,2) -- cycle; 
    \path[ray] (-2,-2) -- (0,0) -- (0,2) -- (0,0) -- (2,0);
    \node at (1.05,1) {$y_1+y_2$};
    \node at (-1,0.3) {$0$};
    \node at (0.7,-1) {$y_1$};
  \end{tikzpicture}
  \quad
  \begin{tikzpicture}[scale=0.6]
   \clip (-2,-2) rectangle (2,2);
   \path[fan] (-2,-2) -- (2,-2) -- (2,2) -- (-2,2) -- cycle; 
   \path[ray] (-2,-2) -- (0,0) -- (-2,2) -- (0,0) -- (2,0);
   \node at (0.8,1) {$y_1+2y_2$};
   \node at (-1,0) {$0$};
   \node at (0.7,-1) {$y_1$};
 \end{tikzpicture}
 \quad
 \begin{tikzpicture}[scale=0.6]
   \clip (-2,-2) rectangle (2,2);
   \path[fan] (-2,-2) -- (2,-2) -- (2,2) -- (-2,2) -- cycle; 
   \path[ray] (-2,0) -- (2,0);
   \node at (0,-0.8) {$0$};
   \node at (0,0.8) {$y_2$};
 \end{tikzpicture}
 \begin{tikzpicture}
   \node at (-0.5,0) {};
   \node at (3.1,0) {$\max\{0,y_1\}$};
   \node at (5.875,0) {$\max\{0,y_1,y_1+y_2\}$};
   \node at (8.65,0) {$\max\{0,y_1,y_1+2y_2\}$};
   \node at (11.5,0) {$\max\{0,y_2\}$};
 \end{tikzpicture}
  \caption{The $g$-vector fan of type $B_2$ (left) is the common refinement of the domains of linearity of the coordinate functions of $\Trop \Psi$ from \Cref{ex:B2trop} after intersecting with the $(y_1,y_2)$-plane.}
  \label{fig:B2trop}
\end{figure}
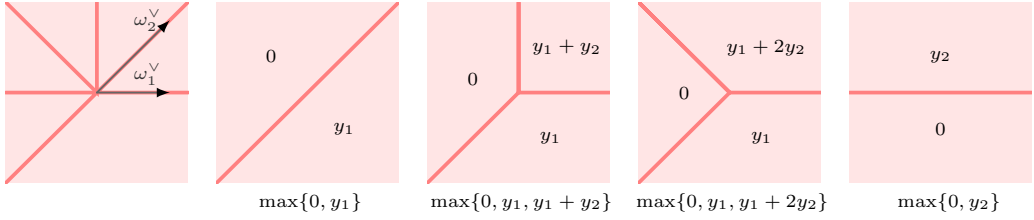

\bibliographystyle{alpha}
\bibliography{JLS2020}

\newcommand{\etalchar}[1]{$^{#1}$}
\begin{thebibliography}{BMDM{\etalchar{+}}18}

\bibitem[AHL20]{AHHL2020}
Nima {Arkani-Hamed}, Song {He}, and Thomas {Lam}.
\newblock {Cluster configuration spaces of finite type}.
\newblock Preprint, available at \url{arxiv.org/abs/2005.11419}, 2020.

\bibitem[BMDM{\etalchar{+}}18]{ABHY-2018}
Véronique Bazier-Matte, Guillaume Douville, Kaveh Mousavand, Hugh Thomas, and
  Emine Yıldırım.
\newblock {ABHY} associahedra and {N}ewton polytopes of {$F$}-polynomials for
  finite type cluster algebras.
\newblock Preprint, available at \url{arxiv.org/abs/1808.09986}, 2018.

\bibitem[BS18]{BS-2018}
Sarah~B. Brodsky and Christian Stump.
\newblock Towards a uniform subword complex description of acyclic finite type
  cluster algebras.
\newblock {\em Algebraic Combinatorics}, 1(4):545--572, 2018.

\bibitem[CFZ02]{CFZ2002}
Frédéric Chapoton, Sergey Fomin, and Andrei Zelevinsky.
\newblock Polytopal realizations of generalized associahedra.
\newblock {\em Canad. Math. Bull.}, 45(4):537--566, 2002.

\bibitem[CLS14]{CLS-14}
Cesar Ceballos, Jean-Philippe Labb\'{e}, and Christian Stump.
\newblock Subword complexes, cluster complexes, and generalized
  multi-associahedra.
\newblock {\em J. Algebraic Combin.}, 39(1):17--51, 2014.

\bibitem[Dem10]{Dem10}
Laurent Demonet.
\newblock Mutations of group species with potentials and their representations.
  {A}pplications to cluster algebras.
\newblock Preprint, available at \url{arxiv.org/abs/1003.5078}, 2010.

\bibitem[DWZ10]{DWZ:2010}
Harm Derksen, Jerzy Weyman, and Andrei Zelevinsky.
\newblock Quivers with potentials and their representations {II}:
  {A}pplications to cluster algebras.
\newblock {\em J. Amer. Math. Soc.}, 23(3):749--790, 2010.

\bibitem[FZ02]{FZ2002}
Sergey Fomin and Andrei Zelevinsky.
\newblock Y-systems and generalized associahedra.
\newblock {\em Canad. Math. Bull.}, 45:537--566, 2002.

\bibitem[FZ03]{FZ2003}
Sergey Fomin and Andrei Zelevinsky.
\newblock Cluster algebras {II}: {F}inite type classification.
\newblock {\em Invent. Math.}, 154:63--121, 2003.

\bibitem[FZ07]{FZ:CA4}
Sergey Fomin and Andrei Zelevinsky.
\newblock Cluster algebras. {IV}. {C}oefficients.
\newblock {\em Compos. Math.}, 143(1):112--164, 2007.

\bibitem[HLT11]{HLT2008}
Christoph Hohlweg, Carsten Lange, and Hugh Thomas.
\newblock Permutahedra and generalized associahedra.
\newblock {\em Adv. in Math.}, 226:206--240, 2011.

\bibitem[HPS18]{HPS2018}
Christoph Hohlweg, Vincent Pilaud, and Salvatore Stella.
\newblock Polytopal realizations of finite type {$\mathbf{g}$}-vector fans.
\newblock {\em Adv. Math.}, 328:713--749, 2018.

\bibitem[KM04]{KM2004}
Allen Knutson and Ezra Miller.
\newblock Subword complexes in {C}oxeter groups.
\newblock {\em Adv. Math.}, 184(1):161--176, 2004.

\bibitem[Lam18]{Lam2018}
Philipp Lampe.
\newblock On the approximate periodicity of sequences attached to
  non-crystallographic root systems.
\newblock {\em Exp. Math.}, 27(3):265--271, 2018.

\bibitem[McM73]{Mullen-1973}
Peter McMullen.
\newblock Representations of polytopes and polyhedral sets.
\newblock {\em Geometriae Dedicata}, 2:83--99, 1973.

\bibitem[PPPP19]{Padrol-Pilaud-2019}
Arnau Padrol, Yann Palu, Vincent Pilaud, and Pierre-Guy Plamondon.
\newblock Associahedra for finite type cluster algebras and minimal relations
  between $\mathbf{g}$-vectors.
\newblock Preprint, available at \url{arxiv.org/abs/1906.06861}, 2019.

\bibitem[PS15]{Pilaud-Stump-2015}
Vincent Pilaud and Christian Stump.
\newblock Brick polytopes of spherical subword complexes and generalized
  associahedra.
\newblock {\em Adv. Math.}, 276:1--61, 2015.

\bibitem[Rea07]{R2007}
Nathan Reading.
\newblock Sortable elements and {C}ambrian lattices.
\newblock {\em Algebra Universalis}, 56(3-4):411--437, 2007.

\bibitem[SW05]{Speyer-Williams-2005}
David Speyer and Lauren Williams.
\newblock The tropical totally positive {G}rassmannian.
\newblock {\em J. Algebraic Combin.}, 22(2):189--210, 2005.

\end{thebibliography}

\end{document}